\documentclass[12pt,dvips]{article}
\usepackage{latexsym}
\usepackage{srcltx}
\usepackage{amsmath,amsthm,amsfonts,amssymb, mathrsfs, wasysym}
\usepackage[T1]{fontenc}
\usepackage[latin1]{inputenc}
\usepackage{aeguill}
\usepackage{url}
\usepackage{enumerate}
\usepackage{graphicx}
\usepackage[a4paper,textwidth=15cm,textheight=25cm]{geometry}

\usepackage{export}

\newtheorem{thm}{Theorem}[section]
\newtheorem{theo}[thm]{Theorem}

\newtheorem*{thm*}{Theorem}
\newtheorem{dfn}[thm]{Definition} 
\newtheorem{defi}[thm]{Definition} 
\newtheorem*{dfn*}{Definition}

\newtheorem{cor}[thm]{Corollary}

\newtheorem*{cor*}{Corollary}

\newtheorem{prop}[thm]{Proposition} 
\newtheorem*{prop*}{Proposition} 
\newtheorem*{properties*}{Properties} 
 
\newtheorem{lem}[thm]{Lemma} 
\newtheorem{lemma}[thm]{Lemma} 
\newtheorem*{lem*}{Lemma}

\newtheorem*{claim*}{Claim} 
 
\newtheorem*{fact*}{Fact}

\newtheorem*{qst*}{Question}

\newtheorem*{pb*}{Problem}

\theoremstyle{remark}
 
\newtheorem*{algo*}{Algorithm} 
\newtheorem*{rem*}{Remark}
\newtheorem{rem}[thm]{Remark}
\newtheorem{remark}[thm]{Remark}
\newtheorem*{example*}{Example}

\newenvironment{SauveCompteurs}[1]{%
\newcommand{\monparametre}{#1}
\openexport{\monparametre_sauve}
  \Export{thm}\Export{section}\Export{subsection}\Export{subsubsection}
\closeexport}{}

\newenvironment{UtiliseCompteurs}[1]{%
\newcommand{\monparametre}{#1}
\openexport{\monparametre_aux}
  \Export{thm}\Export{section}\Export{subsection}\Export{subsubsection}
\closeexport
\PackageInfo{export}{\MessageBreak
Importations from \monparametre_sauve.xpt\MessageBreak}%
\InputIfFileExists{\monparametre_sauve.xpt}{\relax}{\relax}%
\renewcommand{\label}[1]{}
}{%
\PackageInfo{export}{\MessageBreak
Importations from \monparametre_aux.xpt\MessageBreak}%
\InputIfFileExists{\monparametre_aux.xpt}{\relax}{\relax}}

\newlength{\espaceavantspecialthm}
\newlength{\espaceapresspecialthm}
{\normalfont \vskip \espaceapresspecialthm}

\newenvironment{specialthm*}[1]{
\vskip\espaceavantspecialthm \noindent \textbf{#1} \itshape}%
{\normalfont \vskip \espaceapresspecialthm}

\newcommand {\calA} {{\mathcal {A}}}   
\newcommand {\calB} {{\mathcal {B}}}   
\newcommand {\calC} {{\mathcal {C}}}   
\newcommand {\calD} {{\mathcal {D}}}

\newcommand {\calI} {{\mathcal {I}}}   
\newcommand {\calJ} {{\mathcal {J}}}

\newcommand {\calN} {{\mathcal {N}}}

\newcommand {\calU} {{\mathcal {U}}}   
\newcommand {\calV} {{\mathcal {V}}}

\newcommand {\ie}{i.e.\  }

\newcommand {\bbF} {{\mathbb {F}}}

\newcommand {\bbN} {{\mathbb {N}}}

\newcommand {\bbQ} {{\mathbb {Q}}}   
\newcommand {\bbR} {{\mathbb {R}}}

\newcommand {\bbZ} {{\mathbb {Z}}}   

\makeatletter
\edef\@tempa#1#2{\def#1{\mathaccent\string"\noexpand\accentclass@#2 }}
\@tempa\rond{017}
\makeatother

\newcommand{\es}{\emptyset}
\renewcommand{\phi}{\varphi} 
\newcommand{\m} {^{-1}} 
\newcommand{\eps} {\varepsilon}

\newcommand {\ra} {\rightarrow}

\newcommand{\actson}{\curvearrowright}

\newcommand{\ol}[1]{\overline{#1}}

\newcommand{\dunion}{\sqcup}

\newcommand{\grp}[1]{{\langle #1 \rangle}}

\newcommand{\Disc}{\mathrm{d}}

\newcommand{\IET}{\mathrm{IET}}

\newcommand {\frakS}{\mathfrak{S}}

\newcommand{\cald}{\calD}
\newcommand{\calc}{\calC}

\newcommand{\id}{\mathrm{id}}

\newcommand{\supp}{\mathrm{supp}}

\renewcommand{\epsilon}{\varepsilon}

\renewcommand{\int}{\mathrm{int}}

\title{Free groups of interval exchange transformations are rare}
\author{Fran\c{c}ois Dahmani, Koji Fujiwara, Vincent Guirardel}

\begin{document}

\maketitle








\begin{abstract} 
  We study the group $\IET$ of all interval exchange transformations.
Our first main result is that 
the group generated by a generic pairs of elements of $\IET$
is not free (assuming a suitable irreducibility condition on the underlying permutation).
Then we prove that any connected Lie group isomorphic to a subgroup of $\IET$ is abelian.

Additionally, we show that $\IET$ contains no infinite Kazhdan group.
We also prove residual finiteness of finitely presented subgroups of $\IET$ 
and give an example of a two-generated subgroup of $\IET$ of exponential growth
that contains an isomorphic copy of every finite group, and which is therefore not linear.
\end{abstract}

Consider an interval. Break it into finitely many
subintervals. Rearrange these pieces in the interval (preserving the orientation) to obtain
a bijection. This is an interval exchange transformation. 
More precisely, an \emph{interval exchange transformation} of $[0,1)$ is 
a left-continuous bijection, with finitely many discontinuity points, that is piecewise a translation.
The set $\IET$ of all interval exchange transformations is a group for composition.

Interval exchange transformations have been widely studied for themselves, 
in the point of view of the dynamical system they individually  generate (for an introduction, see \cite[Chap. 14.5]{HaKa_introduction}). 
Considerably less is known about the global structure of the group $\IET$.

Basic test questions concern the possible subgroups of $\IET$. 
For instance, the answer to the following question raised by Katok is still unknown.
\begin{qst*}[Katok]
  Does $\IET$ contain a non-abelian free group ?
\end{qst*}

Our first result concerns the subgroups of $\IET$ isomorphic to a connected Lie group, a question raised by Franks.
Using rotations with disjoint support, one can easily realize the $n$-dimensional torus $\mathbb{T}^n$ as a subgroup of $\IET$ for any $n$
(therefore, one can also realize $\mathbb{R}^n$, and any connected abelian Lie group). 
We prove

\begin{UtiliseCompteurs}{Lie}
\begin{theo} 
  A connected Lie groups embeds as a subgroup of $\IET$ if and only if it is isomorphic to 
  $\mathbb{T}^n\times \mathbb{R}^m$.
\end{theo}
\end{UtiliseCompteurs}

The embedding in the statement is purely algebraic, it is not assumed to have any continuity property.
 A similar result was recently independently obtained by Novak \cite{Novak_free, Novak_continuous}.
It is also known that there is no finitely generated subgroup of $\IET$ that has a distorted element for a word metric \cite{Novak_discontinuity}. 
This exclude most nilpotent groups, but does not exclude $SO_3(\bbR)$.
\\

Because translations commute, the orbit of any point under a finitely generated subgroup of $\IET$
has polynomial growth. 
In particular, a possible free subgroup of $\IET$ cannot be produced by controling only the action on a single point,
hence this kind of ping-pong will fail.
Following an observation by Witte-Morris, we show that this implies that 
the image in $\IET$ of any finitely generated Kazhdan group is finite.

But polynomial growth of orbits does not imply polynomial growth of the group 
and we produce an example of subgroup of $\IET$
containing a free semigroup.
\begin{UtiliseCompteurs}{example}
\begin{thm}  
  There is a $2$-generated subgroup $G$ of $\IET$ such that
  \begin{itemize}
  \item $G$ contains a free semigroup
  \item $G$ contains an isomorphic copy of every finite group.
  \end{itemize}
In particular, $G$ is not linear.
\end{thm}
  \end{UtiliseCompteurs}

In the context of connected Lie groups, 
the existence of a free subgroup of a connected lie group $G$
implies that a generic subgroup is free.
Here, a generic subset is a subset containing a countable intersection of dense open subsets. 
Indeed, for each reduced word in two letters and their inverses, the subset  
$V_w\subset G\times G$ of pairs $(a,b)\in G\times G$ with $w(a,b)=1$ is
 a proper analytic subset of $G\times G$ (it is not $G\times G$ because there exists a free subgroup). 
Since $G$ is connected, $V_w$ has empty interior.
Baire's theorem then implies that a generic
subset of $G\times G$ consists of generators of free non-abelian
subgroups.

There is a natural topology on $\IET$ that allows to talk about genericity.
Given a permutation $\sigma$ of  $\{1,\dots,n\}$, the set $\IET_\sigma$ of all interval exchange transformations 
with $n-1$ points of discontinuity and with underlying permutation $\sigma$ is naturally in bijection with the $n$-dimensional open simplex.
This gives a natural topology on $\IET_\sigma$. 
Since $\IET$ is the disjoint union of all $\IET_\sigma$,
this defines a topology on $\IET$ by declaring all $\IET_\sigma$ open in $\IET$.
For this topology, a generic subset of $\IET$ is a subset that intersects each $\IET_\sigma$ in a generic subset.
Note however that this does not make $\IET$ a topological group as the group law is discontinuous.

The genericity argument used above for a connected Lie group does not work in $\IET$.
Indeed, the subset $V_w\subset\IET\times \IET$ is not an analytic subset, but looks like a union of polyhedra,
and can have non-empty interior. 
In fact, we prove that the situation is opposite to that of non-solvable connected Lie groups: 
if one restricts to \emph{admissible} permutations (as defined below), the group generated by a generic pair of elements is not free.
This is a begining of an explanation of the fact that free subgroups of $\IET$  seem difficult to find.

\begin{UtiliseCompteurs}{gnf}
\begin{theo}
Let $\IET_a$ be the set of interval exchange transformations whose underlying permutation is admissible.

Then there is a dense open subset of $\IET \times \IET_a$ such that the group generated by any pair in this subset is not free.
\end{theo}
\end{UtiliseCompteurs}

In this statement, a permutation of $\{1,\dots, n\}$ is \emph{admissible} if there is no $m$ such that $\sigma(m)=m$, and $\{1,\dots, m-1\}$ 
$\sigma$-invariant.
The subset $\IET_a\subset \IET$ is the (open) set of all transformations in $\IET$ with admissible underlying permutation.
\\

Finally, we prove a simple additional result for subgroups of $\IET$ that excludes the existence of Thompson groups
in $\IET$.
\begin{UtiliseCompteurs}{RF}
  \begin{thm}
    Any finitely presented subgroup of $\IET$ is residually finite.

    More generally, any finitely generated subgroup is a limit of
    finite groups in the space of marked groups.  In particular,
    $\IET$ is sofic.
  \end{thm}
\end{UtiliseCompteurs}

The methods of this paper are rather elementary, and the paper is almost
self-contained. 
\\

Our first observation, similar to that in \cite{Novak_discontinuity},
 allows to establish that  elements admitting roots of arbitrarily high order are conjugated to so called \emph{virtual multi-rotations}. 
Denoting by $\Disc(h)$ the number of discontinuity points of an IET $h$, define its growth rate $||h||=\lim_{n\ra\infty} \frac1n\Disc(h^n)$.
Obvious properties are that that $||h||$ is a conjugacy invariant, $||h||\leq\Disc(h)$, and $||h^k||=k||h||$ for $k\geq 0$.
We prove that if one allows to change $[0,1)$ to a more complicated domain (a union of circles and intervals),
then there any IET has a conjugate that is optimal with respect to the growth of its number of discontinuity points:
\begin{UtiliseCompteurs}{discont}
\begin{prop}
  For any $h\in \IET$, there exists another domain $\cald_m$, and an IET $h_m$ conjugate to $h$ such that
$\Disc(h_m)=||h_m||=||h||$.
\end{prop}
\end{UtiliseCompteurs}
This implies that $||h||$ is an integer. In particular, any divisible element satisfies $||h||=0$, and is
therefore conjugate to a continuous IET: a \emph{virtual multi-rotation}.
This can be used to compute centralizers as in \cite{Novak_discontinuity} (see Corollary \ref{cor_centralizer}).

We then argue that no non-abelian free group contains a multi-rotation. In fact, for every multi-rotation $S$, 
and for every interval exchange transformation $T$, 
we can play a game of computing commutators to obtain an interval exchange transformation with small support 
in the generated subgroup $\langle S,T \rangle$. Then, one can conjugate this element to an element with disjoint support.
This exhibits a commutation relation that 
 prevents $\langle S,T \rangle$   from being free non-abelian.

Since connected Lie groups have mostly elements admitting roots of arbitrarily high order, if they are realized as subgroup of $\IET$ they cannot contain any non-abelian free group, therefore they must be solvable. This already rules out groups like $SO_3(\bbR)$. 
For solvable connected Lie groups, we investigate the metabelian case to conclude that they must be abelian.

In order to address our genericity theorem, 
we argue that interval exchange transformations whose subdivision points are well approximated by rational 
numbers almost behave like those ones with rational subdivision points. 
The latter are of finite order, and this allows, again, using commutators, to obtain elements with small supports. 
The next step, again, is to find elements with disjoint supports, to deduce commutation relations.   
This is where we need the assumption that  the underlying permutation of at least one of the elements is admissible.

The paper is organized as follows.
Section \ref{sec_prelim} is mainly devoted to definitions concerning IETs.
Section \ref{sec_discont} constructs our minimal model for discontinuity points.
Section \ref{sec_rotation} shows that multi-rotations cannot be contained in a free group.
Section \ref{sec_Lie} classifies which Lie groups are subgroups of $\IET$.
Section \ref{sec_generic} shows that generically, pairs of elements do not generate a free group.
Sections \ref{sec_lattice} and \ref{sec_RF} show that infinite Kazhdan groups 
and Thompson groups are not subgroups of $\IET$.
Section \ref{sec_example}  explains the example of a $2$-generated subgroup of exponential growth containing all finite groups.

 We would like to thank A.\ Katok, D.\ Calegari and D.\ Witte-Morris for related discussions.
We also thank the referee for useful suggestions.
The second author is in part supported by 
Grant-in-Aid for Scientific Research (No. 19340013).
He also acknowledges the hospitality 
of Institut de Math\'ematiques de Toulouse.


\section{Preliminaries}
\label{sec_prelim}

\subsection{Interval exchange transformations}
A \emph{domain} $\cald$
 is a disjoint finite union of circles of the form $\bbR/l\bbZ$ and of semi-open bounded intervals $[\alpha,\beta)$.
Each component of $\cald$ carries a natural metric, and an orientation.

An \emph{interval exchange transformation} $h:\cald\ra \cald'$ 
between two domains $\cald,\cald'$ is a bijective  
map $h:\cald\ra \cald'$ 
which is piecewise isometric, orientation preserving, 
 continuous on the right, and with only finitely many discontinuity points.
We denote by $\IET(\cald,\cald')$ the set of interval exchange transformation 
from $\cald$ to $\cald'$.
Interval exchange transformations may be composed, and 
 the set $\IET(\cald,\cald)$ is a group which we denote by $\IET(\cald)$
(in fact, the set of all interval exchange transformations has a structure of a groupoid, but we won't make use of 
this terminology). 
We also define $\IET=\IET([0,1))$

The \emph{continuity intervals} of $h$ are the maximal connected subsets of $\cald$ on which $h$ is continuous.
These are finitely many semi-open intervals or circles that partition $\cald$.

We say that $g\in \IET(\cald)$ and $g'\in \IET(\cald')$ are \emph{conjugate} if there exists $h\in \IET(\cald,\cald')$ with $g'=hgh\m$.
There exists an interval exchange transformation $h$ between $\cald$ and $\cald'$ if and only if $\cald,\cald'$ have the same total length.
In this case, conjugation by $h$ induces an isomorphism between $\IET(\cald)$ and $\IET(\cald')$.
Note that even if $\cald$ and $\cald'$ don't have the same total length, conjugation by a homothety shows that $\IET(\cald)\simeq\IET(\cald')$.

A \emph{subdomain} of $\cald$ is a subset of $\cald$ which is a finite union of semi-open intervals and circles.
Interval exchange transformations map subdomains to subdomains.
We define the \emph{support} of $g\in \IET(\cald)$ as $\supp(g)=\{x\in \cald| g(x)\neq x\}$
(and not its closure). It is a subdomain of $\cald$, and is natural under conjugation: $\supp(hgh\m)=h(\supp(g))$.

\subsection{Multi-rotations}

If $\cald$ is a circle, a continuous interval exchange is a \emph{rotation}.
We define a \emph{multi-rotation} as an interval exchange transformation of some domain $\cald$,
which preserves each component on $\cald$,
and which restricts to a rotation on each circle of $\cald$, and to the identity on each segment.
More generally, a \emph{virtual multi-rotation}
is any \emph{continuous} element of $\IET(\cald)$.
Clearly, multi-rotations are virtual multi-rotations, and any virtual multi-rotation has a power which is
a multi-rotation.

Say that a multi-rotation $R$ is \emph{irrational} if
$R$ has infinite order, and
the restriction of $R$ to each circle 
 is either the identity or has infinite order (\ie is an irrational rotation).
Note that any multi-rotation of infinite order has a power which is an irrational multi-rotation.
If $R$ is conjugate to an irrational multi-rotation, then $R$ defines a free action of $\bbZ$ on the support of $R$.

The following is probably well known.

 \begin{lem}\label{lem_conjugant}
   Let $R$ be an irrational rotation on a circle $\calc$.
   \begin{enumerate}
   \item If $R'$ commutes with $R$, then $R'$ is a rotation
   \item If $S\in \IET(\calc)$ is such that $SRS\m$ commutes with $R$,
then $S$ is itself a rotation, and in particular, $S$ commutes with $R$.
   \end{enumerate}
 \end{lem}

 \begin{proof}
For $R'\in \IET$, denote by $\Delta(R')$ its set of discontinuity points.
   Since $R,R\m$ are continuous, $\Delta(R R' R\m)=R(\Delta(R'))$. 
In particular, if $R'$ commutes with $R$, $\Delta(R')$ is $R$-invariant.
Being finite, $\Delta(R')$ has to be empty and $R'$ is a rotation.
This proves the first assertion.

If $R'=SRS\m $ commutes with $R$, then $R'$ is a rotation.
By continuity of $R$ and $R'$, the relation $R'S=SR$ yields
$\Delta(S)=R\m\Delta(S)$ which implies $\Delta(S)=\es$ as above.
 \end{proof}






\subsection{Irrational circles}

 \begin{dfn}
   Let $T\in\IET(\cald)$. An \emph{irrational circle} of $T$ is a subdomain $\calc\subset \cald$
(not necessarily homeomorphic to a circle, or even connected), which is $T$-invariant, and such that
$T_{|\calc}$ is conjugate to an irrational rotation on a circle.
 \end{dfn}

For example, given $0<\tau<l<1$, consider the transformation $h:[0,1)\ra[0,1)$
defined as the translation of length $\tau$ on $[0,l-\tau)$,
as the translation of length $\tau-l$ on $[l-\tau,l)$, and as the identity on $[l,1)$.
Then $[0,l)$ is an irrational circle of $h$ if and only if $\tau/l\notin\bbQ$.

Of course, $T$ is conjugate to 
an irrational multi-rotation if and only if its support is a disjoint union of irrational circles.

It follows from the definition that if $g'=hgh\m$ and if $\calc$ is an irrational circle of $g$,
then $h(\calc)$ is an irrational circle of $g'$.

 \begin{lem}\label{lem_minimal}
For any $T\in\IET(\cald)$, two irrational circles of $T$ either are disjoint or coincide.

More generally, if a subdomain $\cald_0\subset \cald$ intersects
some irrational circle $\calc$, and if $T(\cald_0)=\cald_0$, then $\cald_0\supset \calc$.
 \end{lem}

 \begin{proof}
The first statement follows from the second.
Assume that there is some point $x\in \cald_0\cap\calc$.
Then $\cald_0\cap\calc$ contains a small interval $I=[x,x+\eps)$.
Since $\cald_0$ is $T$-invariant, $T^i(I)\subset \cald_0$ for all $i\in\bbZ$.
Since $T_{|\calc}$ is conjugate to an irrational rotation, there exists $i$ such that $\calc=I\cup T(I)\dots\cup T^i(I)$,
and the lemma follows.
 \end{proof}

\begin{lem}\label{lem_coincide}
  Let $S,T\in \IET(\cald)$ be conjugate to irrational multi-rotations.

If $S,T$ commute, then for each irrational circle $\calc$ of $S$, either $\calc$ does not intersect $\supp(T)$,
or $\calc$ is an irrational circle of $T$.
\end{lem}

\begin{proof}
  Since $S$ and $T$ commute, $T(\calc)$ is an irrational circle of $S$. 
Since $S$ has only finitely many irrational circles, there exists $k\geq 1$ such that
$T^k(\calc)=\calc$. 
Assume that $\calc$ intersects $\supp(T)$, \ie that it intersects some irrational circle $\calc_T$ of $T$.
By Lemma \ref{lem_minimal} (applied to the transformation $T^k$ and to $\cald_0=\calc$), $\calc_T\subset \calc$.
By symmetry of the argument, $\calc\subset \calc_T$. The lemma follows.
\end{proof}

\begin{lem}\label{lem_csa}
   Let $S$ and $T$ be conjugate to irrational multi-rotations. 
         If for all $n$,  $S^{n}TS^{-n}$  commutes with $T$,   then $S$ commutes with $T$.
\end{lem}

\begin{proof}
We prove that for each irrational circle $\calc_T$ of $T$, $S$ preserves $\calc_T$ and $S_{|\calc_T}$ commutes with $T_{|\calc_T}$.
The lemma follows immediately.

We can assume that $\calc_T$ intersects the support of $S$, hence some irrational circle $\calc_S$ of $S$.
Let $I=[x,x+\eps)$ be a segment in $\calc_S\cap\calc_T$.
There is $n\geq 1$ such that $S^n(I)\cap I\neq\es$, so $S^n(\calc_T)$ intersects $\calc_T$.
Since $S^n(\calc_T)$ is an irrational circle of $S^{n}TS^{-n}$ which commutes with $T$, 
 $S^n(\calc_T)=\calc_T$ by Lemma \ref{lem_coincide}.

Applying Lemma \ref{lem_minimal} to the $S^n$-invariant subdomain $\calc_T$,
we get $\calc_T\supset \calc_S$.
Since this holds for all irrational circle $\calc_S$ which intersect $\calc_T$,
$\calc_T$ is $S$-invariant.
Since $T_{|\calc_T}$ is conjugate to an irrational rotation,
we can apply Lemma \ref{lem_conjugant} to the restrictions of $T$ and $S$ to $\calc_T$,
and  get that $S_{|\calc_T}$  commutes with $T_{|\calc_T}$. 
\end{proof}

\section{Points of discontinuity}
\label{sec_discont}
\subsection{Definitions}

\begin{dfn}
  Given $h:\cald\ra \cald'$ an interval exchange map, we denote by 
$\Delta(h)$ the finite set
of points in $\cald$ where $h$ is discontinuous,
and by $\Disc(h)=\#\Delta(h)$ its cardinal.
\end{dfn}

We define the \emph{interior} $\rond \cald$ of $\cald$ as the set of points $x$ having a neighbourhood isometric to an open interval $(-\eps,\eps)$.
We note that $h$ being continuous on the right, it is automatically continuous at points in $\cald\setminus \rond \cald$.
Conversely, if $h(x)\in \cald\setminus \rond \cald$ and $x\in\rond{\cald}$, then $h$ is not continuous at $x$.

\begin{lem}\label{lem_subadd}
  Consider $h\in \IET(\cald,\cald')$,  $h'\in \IET(\cald',\cald'')$. 

Then $\Disc(h'\circ h)\leq \Disc(h)+\Disc(h')$. 
\end{lem}

\begin{proof}
For any  $x\in \Delta(h'\circ h)$, then either $x\in \Delta(h)$ or $h(x)\in \Delta(h')$
as the composition of continuous maps is continuous.
Thus $\Delta(h'\circ h)\subset \Delta(h)\cup h\m \Delta(h')$.
Since $h$ is a bijection, the lemma follows.
\end{proof}

\begin{lem}
  Given $h\in \IET(\cald)$, we can define
$$||h||=\inf \{\frac{1}{n} \Disc(h^n)|n\in\bbN\}=\lim_{n\ra \infty}\frac{1}{n} \Disc(h^n).$$
It satisfies $||h^k||=k||h||$ for all $k\in \bbN$, and 
if $h,h'$ are conjugate, then $||h||=||h'||$.
\end{lem}

\begin{proof}
The limit exists and coincides with the infimum by subadditivity (Lemma \ref{lem_subadd}).
For any $h\in \IET(\cald)$, $||h^k||= \lim_{n\ra \infty}\frac{1}{n} \Disc(h^{kn})=k\lim_{n\ra \infty}\frac{1}{kn} \Disc(h^{kn})=k||h||$. 
If $h'=ghg\m$, then $h'^n=gh^ng\m$ so $\Disc(h'^n)\leq \Disc(g)+\Disc(h^n)+\Disc(g\m)$.
Passing to the limit, we get $||h'||\leq ||h||$. Since the symmetric inequality holds, the lemma follows.
\end{proof}

\subsection{Minimal model for discontinuity points, and applications}

\begin{SauveCompteurs}{discont}
\begin{prop}\label{prop_disc}
  Let $h\in \IET(\cald)$. Then there exists another domain $\cald_m$ and $h_m\in IET(\cald_m)$
conjugate to $h$, such that $\Disc(h_m^n)=n\Disc(h_m)$ for all $n\in \bbN$.
\end{prop}
\end{SauveCompteurs}

This conjugate $h_m$ has the minimal number of discontinuity points among all conjugates of $h$ as shows the third assertion
of the following corollary.

\begin{cor}\label{cor_growth}
Consider any $h\in \IET(\cald)$. Then
  \begin{enumerate}
  \item  $||h||=||h_m||=\Disc(h_m)$. In particular, $||h||\in \bbN$.

  \item There exists a constant $C$ such that for all $n\geq 0$, $n||h||-C\leq \Disc(h^n)\leq n||h||+C$

  \item If $h'$ is any other conjugate of $h$, then $\Disc(h_m^n)\leq \Disc(h'^n)$ for all $n$.

  \item  \label{item_vmr} $||h||=0$, if and only if $h$ is conjugate to a virtual multi-rotation

  \item \label{item_divis} If $h$ is divisible (i.e. if $h$ has roots of arbitrary order), then $||h||=0$.
  \end{enumerate}

\end{cor}

\begin{proof}
The first assertion is clear.

Let $g\in\IET(\cald,\cald_m)$ conjugating $h$ to $h_m$. Then 
$\Disc(h^n)\leq \Disc(g)+\Disc(h_m^n)+\Disc(g\m)= \Disc(g)+n||h||+\Disc(g\m)$.
Similarly, $n||h||=\Disc(h_m^n)\leq \Disc(g)+\Disc(h^n)+\Disc(g\m)$. Assertion 2 follows.

For assertion 3, assume $\Disc(h'^{n_0})<\Disc(h_m^{n_0})=n_0||h||$.
Then $||h'||=\inf\{\frac1n \Disc(h'^{n})\}< ||h||$, a contradiction.

If $||h||=0$, then $\Disc(h_m)=0$, so $h_m$ is continuous, i.e.,
a virtual multi-rotation. Hence $h$ is conjugate to a virtual multi-rotation. The converse is obvious.

Finally, assume that $||h||>0$. Since $||g^k||=k||g||$, and since $||\cdot||$ takes only integer values,
$h$ has no root of order larger that $||h||$.
\end{proof}

The following lemma shows that the orbit of $\Delta(h_m)$ is canonical.

\begin{lem}
  Let $h_m$ be such that $d(h_m)=||h_m||$.
Then the centralizer of $h_m$ 
preserves $h_m^\bbZ(\Delta(h_m))$. Equivalently,
if $[g,h_m]=1$, then for any $x\in\Delta(h_m)$ there exists $k\in\bbZ$ such that
$g(x)\in h_m^k(\Delta(h_m))$.
\end{lem}

\begin{proof}
Fix $x\in \Delta(h_m)$.
Since $||h_m||=d(h_m)$,
for all $i>0$, $d(h_m^i)=id(h_m)$, so $\Delta(h_m^i)$ 
is the disjoint union $\Delta(h_m)\dunion\dots \dunion h_m^{-(i-1)}.\Delta(h_m)$.
In particular, $x,\dots, h_m^{-(i-1)}(x)\in\Delta(h_m^i)$.

  Since $g$ commutes with $h_m^i$ for all $i>0$, 
$$\Delta(h_m^i)=\Delta(g\m h_m^ig)\subset \Delta(g)\cup g\m.\Delta(h_m^i)\cup g\m h_m^{-i}.\Delta(g\m).$$
In particular, $\Delta(h_m^i)\setminus g\m.\Delta(h_m^i)$ contains at most $2d(g)$ elements.
It follows that for $i>2d(g)$, there is some $i_0\in\{0,\dots, i-1\}$ such that $h_m^{-i_0}(x)\in g\m.\Delta(h_m^i)$.
Hence, $g h_m^{-i_0}(x)$ lies in some $h_m^{-j}.\Delta(h_m)$ for some $j\in\{0,\dots,i-1\}$,
and the lemma follows since $g$ and $h_m$ commute.
\end{proof}

These considerations allow to slightly simplify Novak's result on centralizers.

\begin{cor}[\cite{Novak_discontinuity}]\label{cor_centralizer}
  It $h\in\IET$ acts with dense orbits on $\cald$, and $||h||\neq 0$, then the centralizer of $h$ is virtually cyclic.
\end{cor}

\begin{proof}
By Proposition \ref{prop_disc}, we may assume that $||h||=d(h)$. Denote by $Z$ the centralizer of $h$.
By the previous lemma, $Z$ permutes the $h$-orbits of the points of $\Delta(h)$. 
Let $Z_0\subset Z$ be the finite index subgroup
that preserves the $h$-orbit of each point of $\Delta(h)$. 
We claim that $Z_0=\grp{h}$. Indeed, consider $g\in Z_0$, $x\in \Delta(h)$,
and $k$ such that $g(x)=h^k x$. Since $g$ commutes with $h$, for all $i\in\bbZ$, $g(h^i(x))=h^{k}(h^ix)$.
Since the orbit of $x$ is dense, $g=h^k$.
\end{proof}

The proof of Proposition \ref{prop_disc} is better visualized in terms of suspensions. 
It is a direct consequence of Lemmas \ref{lem_hS_exists} and \ref{lem_hS_is_min}, whose proofs occupy next subsection.

\subsection{Suspensions and construction of minimal models}

Consider $h\in \IET(\cald)$. We now define its suspension (in a
slightly non-standard way).
Let $I_1,\dots,I_k$ be the continuity intervals of $h$, \ie the maximal connected sets on which $h$ is continuous.
Each $I_i$ is either a circle, or a semi-open interval, closed on the left.
Let $\ol \cald$ be the metric completion of $\cald$, and $\ol I_i$ the closure of $I_i$ in $\ol \cald$.
Note that $h_{|I_i}$ extends by continuity to a map $\ol h_i:\ol I_i\ra\ol \cald$.
We denote by $J_i=h(I_i)$, which are the maximal connected sets on which $h\m$ is continuous
and by $\ol J_i=\ol h_i(\ol I_i)$ its closure in $\ol \cald$.

We consider \emph{bands} $B_i=\ol I_i\times [0,1]$, which we glue on $\ol \cald$ using 
two maps $\phi_\eps:\ol I_i\times \{\eps\}\ra \ol \cald$ defined as follows:
$\phi_0:(x,0)\in \ol I_i\times\{0\}\mapsto x\in \ol I_i$ is just the inclusion,
and $\phi_1(x,1)\in \ol I_i\times\{1\}\mapsto  \ol h_i(x)\in\ol J_i$ is the extension of $h_{|I_i}$.
Since $\ol I_i$ can be a circle, bands can be annuli.

We denote by $\Sigma(h)$ (or simply $\Sigma$) the cellular complex obtained in this way.
We call it the \emph{suspension} of $h$.
If we foliate bands by $\{x\}\times [0,1]$, $\Sigma$ inherits a natural foliation such that
for all $x\in \cald$, $x$ lies in the same leaf as $h(x)$.

Since $I_1 \dunion \dots\dunion I_k$ is a partition of $\cald$,
every $x\in \cald$ lies in exactly one interval $\ol I_i$ except if $h$ is discontinuous at $x$, in which case $x$ lies in the intersection
of exactly two closed intervals $\ol I_i$.
Similarly, any $x\in \cald$ lies in exactly one $\ol J_i$ except if $h\m$ is discontinuous at $x$.

\begin{figure}[htbp]
  \centering
  \includegraphics{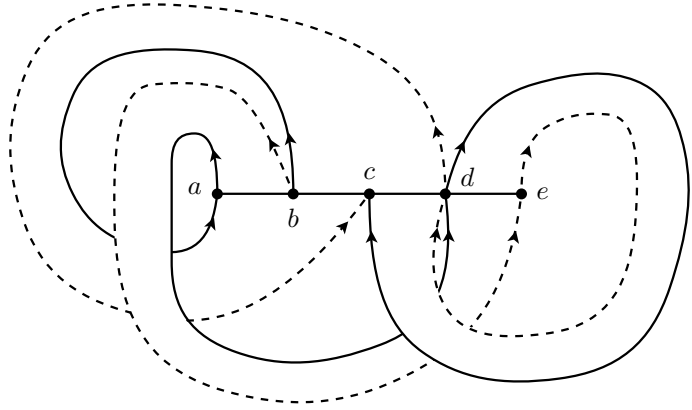}
  \caption{Suspension of an interval exchange transformation}
  \label{fig_susp}
\end{figure}

\newcommand{\Sing}{\mathrm{Sing}}

Denote by $\Sing(h)=\Delta(h)\cap \Delta(h\m)$ the set of points where both $h$ and $h\m$ are discontinuous
(see Figure \ref{fig_susp} where $\Delta(h)=\{b,d\}$, $\Delta(h\m)=\{c,d\}$ and $\Sing(h)=\{d\}$).
On Fig.\ \ref{fig_susp}, dotted lines correspond to leaf segments $x\times [0,1]$ 
whose endpoints are not the image of each other under $h$ (only under $\ol h_i$).

Points in $\cald\setminus (\Delta(h)\cup \Delta(h\m))$ (resp.\ $\Delta(h)\cup \Delta(h\m)\setminus \Sing(h)$)
 have an neighbourhood in $\Sigma$ homeomorphic to $\bbR^2$ (resp.\ to a $\bbR_+\times\bbR$). 
In particular, $\Delta(h)$ is a surface with boundary if (and only if) $\Sing(h)=\es$.
In this case, $\Delta(h)\cup\Delta(h\m)\subset \partial\Sigma$.

It will be useful to perform some moves that change $h$ to a conjugate, and that allow to simplify $\Sigma$.

Consider $h\in \IET(\cald)$, and $x\in \rond{\cald}$. 
Let $\cald'$ be the domain obtained from $\cald$ by cutting at $x$.
More precisely, if the component of $\cald$ containing $x$ is $[\alpha,\beta)$ with $\alpha<x<\beta$,
we replace $[\alpha,\beta)$ by the two intervals $[\alpha,x)$ and $[x,\beta)$.
Similarly, if the component of $\cald$ containing $x$ is a circle $\bbR/l\bbZ$,
we replace it by an interval $[x,x+l)$.
We denote by $i:\cald'\ra \cald$ the natural inclusion ($i\in\IET(\cald',\cald)$), and
define $h'=i'{}\m\circ h\circ i\in\IET(\cald')$. We say that $h'$ is obtained from $h$ by splitting $\cald$ at $x$.

\begin{lem}\label{lem;no_sing}
  Every $h\in \IET(\cald)$ is conjugate to some $h'\in \IET(\cald')$ such that $\Sigma(h')$ is a surface with boundary.
\end{lem}

\begin{proof}
Assume that $\Sigma(h)$ is not a surface and
consider a point $x\in \Sing(h)$.
Let $h'\in \IET(\cald')$ be obtained by splitting $\cald$ at $x$.
We claim that $\#\Sing(h')=\#\Sing(h)-1$. The Lemma will follow by induction.

Denote by $i:\cald'\ra \cald$ the splitting map, and $x'=i\m(\{x\})$
 the copy of $x$ in $\cald'$.
Since $i$ is continuous on $\cald'$ and $i\m$ is only discontinuous at $x$,
$\Sing(h)\setminus x=i(\Sing(h')\setminus\{x'\})$.
Since $x'\in  \cald'\setminus\rond \cald'$, $h'$ and $h'{}\m$ are continuous at $x'$.
The claim follows.
\end{proof}

When $\Sigma$ is a surface, one can describe its boundary $\partial\Sigma$.
Each connected component of $\partial \Sigma$ is a circle $C$, which has a natural set of vertices $C\cap \cald$, 
and a natural set of edges contained in boundaries of bands $(\ol I_i\setminus \rond{I_i})\times [0,1]$. 
Each edge $\{x\}\times [0,1]$ of $C$ carries a preferred orientation from $0$ to $1$.
A vertex $x\in C$ can be of three exclusive types: 
$x\in \ol \cald\setminus \rond{\cald}$ having one incoming and one outgoing edge 
(the vertices $a,e$ in Fig.\ \ref{fig_susp}),
$x\in \Delta(h)$ having two outgoing edges (the vertex $b$),
and $x\in \Delta(h\m)$ having two incoming edges
(the vertex $c$. Recall that $\Sing(h)=\Delta(h)\cap\Delta(h\m)=\es$ by assumption).

We say that $\Sigma$ has a \emph{boundary connection} 
if there exists a leaf segment in $\Sigma$  which intersects  $\partial\Sigma$
exactly at its endpoints, or if $\Sing(h)\neq \es$. Equivalently, $\Sigma(h)$ has a boundary connection if
 there exists $x\in \Delta(h\m)$ and $k\geq 0$ such that $h^k(x)\in \Delta(h)$. 
Note that taking $k$ as small as possible, this is equivalent to ask that
$h\m$ discontinuous at $x$ but continuous at $h(x),h^2(x),\dots, h^k(x)$ and
$h$ discontinuous at $h^k(x)$ but continuous at $x,h(x),\dots, h^{k-1}(x)$.

\begin{lem}\label{lem;no_BC}
  Every $h\in \IET(\cald)$ is conjugate to some $h'\in \IET(\cald')$ such that $\Sigma(h')$ has no boundary connection. 
\end{lem}

\begin{proof}
We can assume that $\Sing(h)=\es$ by Lemma \ref{lem;no_sing}. 
  Assume that $x,h(x),h^2(x),\dots, h^k(x)$ is a boundary connection, with $x\in \Delta(h\m)$,
$h^k(x)\in \Delta(h)$, and $h(x),\dots h^{k-1}(x)\in \rond \cald\setminus (\Delta(h)\cup \Delta(h\m))$.
We split $\cald$ at the points $x,h(x),\dots, h^{k}(x)$, and denote by $\cald'$ the new domain,
and by $i:\cald'\ra \cald$ induced by inclusion.
We denote by $x_i=h^i(x)\in \cald$, and $x'_i=i\m(h^i(x))\in \cald'$ for $i=0,\dots, k$.
Since $x'_i\in  \cald'\setminus \rond \cald'$, $\Delta(h')$ does contains no $x'_i$. 
Moreover, since $i$ and $i\m$ are continuous away from $x_i,x'_i$, 
$i\m(\Delta(h)\setminus \{x_0,\dots,x_k\})=\Delta(h')\setminus  \{x'_0,\dots,x'_k\}=\Delta(h')$,
and similarly for $\Delta(h)\m$. We still have $\Sing(h')=\es$ and $\Sigma(h')$ has one less boundary connection.
By induction, the lemma follows.
 \end{proof}

Now there are some useless boundary components which we want to get rid of.
They correspond to points $x$ such that $h$ is discontinuous at $x$, but some positive power 
is continuous at $x$.
We will denote by $h(x^-)=\lim_{\eps\ra 0^+} h(x-\eps)$ (only defined if $x\in \rond \cald$), and by $h(x^+)=\lim_{\eps\ra 0^+} h(x+\eps)=h(x)$.

Let $C\subset \partial\Sigma$ be a boundary component.
We say that $C$ is a \emph{fake boundary} if $C\cap \rond \cald$ consists of exactly two vertices $x,y$, 
(necessarily one in $\Delta(h)$, the other in $\Delta(h\m)$), 
and the two connected components of $C\setminus \{x,y\}$ have the same number of edges.
Denote by $k$ this number of edges, and up to exchanging $x$ and $y$, assume $x\in C\cap \Delta(h)$. 
Then  $h^k(x)$ coincides with $h^k(x^-)$, so $h^k$ is continuous at $x$ although $h$ is not.

\begin{lem}\label{lem_fake}
  Assume that $\Sigma(h)$ has no boundary connection, and that $x\in \Delta(h)$ is such that $h^k$ is continuous $x$
for some $k>0$. Then  $\Sigma(h)$ has a fake boundary.
\end{lem}

\begin{proof}
  Define $x_i^+=h^i(x)=h^i(x^+)$ and  $x_i^-=h^i(x^-)$,
and let $k>0$ be smallest such that $x_k^+=x_k^-$. Since $x\in \Delta(h)$, $k\geq 2$.

We claim that for all $i\leq k$, $x_i^+$ and $x_i^-$ lie in $\partial \Sigma(h)$.
Otherwise, consider some $x_i^+$ (resp.\ $x_i^-$) which does not lie in $\partial \Sigma(h)$.
Then because there is no boundary connection, $x_k^+$ (resp.\ $x_k^-$) does not lie in $\partial \Sigma(h)$, so $h\m$ is continuous at $x_k^+=x_k^-$.
It follows that $x_{k-1}^+=x_{k-1}^-$, a contradiction.

It follows that the component $C$ of $\partial \Sigma$ containing $x$ contains oriented edges 
joining $x_i^+$ to $x_{i+1}^+$ (resp. $x_i^-$ to $x_{i+1}^-$) for $i=0,\dots,k-1$,
and $C$ is a fake boundary.
\end{proof}

 \begin{lem}\label{lem_hS_exists}
  Every $h\in \IET(\cald)$ is conjugate to some $h_m\in \IET(\cald_m)$ such that $\Sigma(h_m)$ has no
boundary connection and no fake boundary.    
 \end{lem}

 \begin{proof} First, by Lemma \ref{lem;no_BC} we can assume that $\Sigma(h)$ has no boundary connection. We will produce a sequence of conjugates of $h$ whose suspensions have no boundary connection, 
and 
 fewer fake boundaries. By finiteness of the number of boundaries, this will ensure the result. 

   Assume that $C$ is a fake boundary,
denote by $C_l,C_r$ the two components of $C\setminus (C\cap\rond \cald)$,
 $k$ their number of edges,  and consider $x\in C\cap \Delta(h)$.
Note that $k\geq 2$ since otherwise, $h$ would be continuous at $x$, a contradiction.
Introduce $x_i=h^i(x)\in C_r$ for $i=\{1,\dots, k-1\}$, 
and $y_i=h^i(x^-)\in C_l$. Note that $x_i$ is a (left) endpoint of $\cald$, but $y_i\in \ol \cald\setminus \cald$.

We now perform a gluing move on $h$ as follows.
Let $\cald'$ be obtained from $\cald\cup \{y_1,\dots,y_{k-1}\}\subset \ol \cald$ by identifying 
$y_i$ with $x_i$. Clearly, $\cald'$ is a domain, and the inclusion $j:\cald\ra \cald'$ allows to
define $h'=jhj\m \in \IET(\cal D')$. 

This is a general construction, and we claim that if $\Sigma(h)$ has no boundary connection, then neither does $\Sigma(h')$, indeed 
assume $x,h'(x),\dots,(h')^k(x)$ is a boundary connection in
$\Sigma(h')$, then,  if no $j^{-1}((h')^i(x))\in C$ ($i=0,\dots,k$), 
its image under $j^{-1}$ would also be one in $\Sigma(h)$; and
if some $j^{-1}((h')^i(x))\in C$, then a shorter path would provide a boundary connection in $\Sigma(h)$.

Since $\Sigma(h')$ has one less fake boundary, the lemma follows by induction.
  \end{proof}

\begin{lem}\label{lem_hS_is_min}
 Let $h_m$ be as above. 
Then for all $n\in\bbN$, $\Disc(h_m^n)=n\Disc(h_m)$.
\end{lem}

\begin{proof}
We can assume $n>0$. 
By Lemma \ref{lem_subadd}, one has $\Disc(h_m^n)\leq n\Disc(h_m)$.

Let $\Delta\subset \cald$ be the set of discontinuity points of $h_m$.
Recall that $\Delta$ lies in the boundary of the suspension $\Sigma$ of $h_m$.
Let $x\in \Delta$. 
By Lemma \ref{lem_fake}, $h_m^k$ is discontinuous at $x$ for all $k>0$.
We claim that $h_m^n$ is discontinuous at each point $h_m^{-k}(x)$ for $k=0,\dots, n-1$.
Indeed, since there is no boundary connection, $h_m^{-i}(x)\notin \partial \Sigma$  for all $i>0$, 
so $h_m\m$ is continuous at all these points. 
If $h_m^n$ was continuous at $h_m^{-k}(x)$ $0\leq k<n$, then $h_m^n\circ (h_m\m)^k=h_m^{n-k}$ would be continuous at $x$, a contradiction.

We proved that $h_m^n$ is discontinuous on $\bigcup_{i=0}^n h_m^{-i}(\Delta)$.
We claim that this is a disjoint union. Indeed,
 the absence of boundary connection implies that $h_m^{-i}(x)$ is not in $\partial \Sigma$ hence not in $\Delta$ for all $i>0$.
It follows that all $h_m^{-i}(\Delta)$ are disjoint, and the lemma follows.
\end{proof}

Proposition \ref{prop_disc} is proved.

\section{Commutation relations involving multi-rotations}
\label{sec_rotation}
Let $\cald$ be a domain.
For a subset $Y \subset \cald$ and $\epsilon >0$,  let $Y_{\epsilon}$ denote
the closed $\epsilon$-neighborhood of $Y$ in $\cald$, where
we assume that the distance between two components of $\cald$
is infinite.
We define  $\int_\eps(Y)=\{x\in Y | [x-\eps,x+\eps] \subset Y\}$.
By definition, we have $\int_{\epsilon}(\cald\backslash Y) \cap Y_{\epsilon}=\emptyset$.

Recall that the support of $T\in \IET(\cald)$ is
 $\supp(T)=\{x \in \cald|T(x)\neq x\}$.
We denote by $[g,h]=g\m h\m gh$ the commutator of two group elements $g,h$.
The following proposition allows to produce elements with small support.

\begin{prop}\label{prop_support}
  Let $R\in \IET(\cald)$ be a multi-rotation, and consider any $S\in \IET(\cald)$. 
Let $X=(\ol \cald\setminus \rond\cald)\cup \Delta(S)\cup\Delta(S\m)$.
Then for all $\eps>0$, there exists $n\geq 1$ such that the support of $U=[[S,R^n],R^n]$
is contained in $X_{\eps}$, the closed $\eps$-neighbourhood of $X$.
\end{prop}

\begin{rem}\label{rem_bord}
If $\cald$ is a union of circles, then $X=\Delta(S)\cup\Delta(S\m)$. 
\end{rem}

We first prove a few easy lemmas.
We will be concerned with subsets $E$ that consist of a finite union of sub-intervals.

\begin{lemma}\label{lem_small}
  Consider a subset $E\subset \cald$,
 $g\in\IET(\cald)$ such that $g$ is continuous (\ie a translation) on each connected component of $E$. 
Consider $h\in\IET(\cald)$ whose restriction to each component of $E$, and of $g(E)$, is  a {(continuous)} 
translation of { amplitude}  $\in [-\epsilon, \epsilon]$.

Then, on each component of   $\int_{2\epsilon}(E)$, 
the restriction of
 $[g,h]$ is a translation of { amplitude} $\in [-2\epsilon, 2\epsilon]$.
\end{lemma}

\begin{proof}
Let $I$ be a connected component of $E$, and $t,t'\in[-\eps,\eps]$ be such that $h_{|I}(x)=x+t$
and $h_{|g(I)}(x)=x+t'$.
Then for $x\in \int_{2\eps}(I)$, $h(x)=x+t\in \int_\eps(I)$, and since $g_{|I}$ commutes with translations, 
$h\m gh(x)=h\m( g(x)+t)=g(x)+t-t'$. Similarly, since $g(x)+t-t'\in g(I)$, $g\m h\m gh(x)=x+t-t'$.
\end{proof}

As an immediate corollary, we have:

\begin{lemma}\label{lem;commut1}
 Let $g\in\IET(\cald)$ be arbitrary, and $r\in\IET(\cald)$ a multi-rotation which moves all points of $\cald$ by at most $\eps/2$.
 Let $X =\Delta(g)\cup\Delta(g\m)$.

Then  the restriction of $[g,r]$  to each component of 
$\int_{\eps}(\cald \setminus X)$ is a (continuous) translation of amplitude in $ [-\epsilon, \epsilon]$.\qed
\end{lemma}

\begin{lemma}\label{lem;commut1.5}
Let $E\subset\cald$. 
Consider $g, h\in \IET(\cald)$  whose restrictions on  each component of $E$ are (continuous) translations 
of amplitude $\in [-\epsilon, \epsilon]$.  

Then, $[g,h]$ is the identity in restriction to $\int_{\eps}(E)$.
\end{lemma}

\begin{proof}
  Let $x\in \int_{\eps}(E)$. Denote $\alpha,\beta$ the amplitude of the translations induced by $g$ and $h$
in the component of $E$ containing $x$. Since $g(x),h(x)\in E$, $hg(x)=x+\alpha+\beta=gh(x)$.
The lemma follows.
\end{proof}

\begin{proof}[Proof of Proposition \ref{prop_support}]
  Let $R$ be a multi-rotation on a domain $\cald$ with $k$ circles. We view $R$ as an element of the torus group $({\Bbb S}^1)^k$.
Since this group is compact, powers of $R$ get arbitrarily close to the identity.
In other words, for all $\eps>0$, there exists $n$ such that $R^n$ moves all points of $\cald$ by at most $\eps/2$.
By Lemma \ref{lem;commut1}, the restriction of $[S,R^n]$ to each component of $\int_{\eps}(\cald\setminus X)$ 
is a translation of amplitude in $[-\epsilon,\epsilon]$.
By Lemma \ref{lem;commut1.5}, $[[S,R^n],R^n]$ is the identity in restriction to 
$\int_{2\eps}(\cald\setminus X)$.  
The proposition is proved.  
\end{proof}

\begin{thm} \label{thm_no_MR_in_free}
  A  non-abelian free subgroup of $\IET(\cald)$ contains no non-trivial element conjugate to a 
virtual 
multi-rotation. 
\end{thm}

\begin{proof}
We argue by contradiction. 
Consider $R$ conjugate to a non-trivial  virtual multi-rotation, contained in a non-abelian free subgroup $H$ of $\IET$.
Up to conjugating $H$ to a group of IETs on some other domain, we can assume that $R$ is indeed a 
virtual multi-rotation of $\cald$.
Since $R$ has infinite order, we can change $R$ to some power and assume that $R$ is an irrational multi-rotation.
Consider $S\in H$ not commuting with $R$, so that $R,S$ form a basis of the  subgroup they generate (if two elements generate a non-abelian free group,
they are a basis of this free group by Hopf property of free groups).
Define $R'=SRS\m$, and note that $\{R,R'\}$ freely generates a free subgroup of $H$. 
We aim to find a non-trivial relation between $R,R'$ to get a contradiction.

Let $\calc\subset \cald$ be the support of $R$ (a union of circles of $\cald$),
and $\calc'\subset \cald$ be the support of $R'$ (which may be arbitrary).
Up to restricting to a subdomain of $\cald$, 
we can assume that $\calc\cup\calc'=\cald$, and we can still assume that $\cald$ is a union of circles by identifying each interval of $\cald$
(on which $R$ is the identity) to a circle of same length.
Define $X=\Delta(R')\cup\Delta(R'{}\m)$, and $X_\eps$ its closed $\eps$-neighbourhood.

\begin{lem}\label{lem_disj}
There exists $p,p'\geq 1$ and $\eps>0$ such that
$R'^{p'}R^p(X_\eps)\cap X_\eps=\es$.
\end{lem}

\begin{proof}
Since $\calc$ is the support of the irrational multi-rotation $R$, 
there exists $p\geq 1$ such that $R^p(X\cap\calc)\cap X=\es$. 
  Denote by
  $Y_\eps=R^p(X_\eps)=N_\eps(R^p(X))$ since $R$ is continuous.  For $\eps$ small enough, $Y_\eps\cap\calc$ does not intersect
  $X_\eps$. 
  We need to prove that $R'^{p'}Y_\eps\cap X_\eps=\es$ for suitable $p'$ and $\eps$.  Since $R'$ is the identity on $\cald\setminus
  \calc'$, and since $\cald\setminus\calc'\subset \calc$, for any $p'$, $R'^{p'}(Y_\eps)\cap (\cald\setminus \calc')\subset
  Y_\eps\cap\calc$ and does not intersect $X_\eps$.  Thus, we need only to prove that $R'^{p'}(Y_\eps)\cap\calc'$ does not
  intersect $X_\eps$ for suitable $p'$ and $\eps$, or equivalently (since $R'=SRS\m$), that $R^{p'}(S\m (Y_\eps)\cap\calc)$ does
  not intersect $S\m(X_\eps)$.  Now there exist finite sets $X',Y'\subset \cald$ such that, for all $\eps>0$ small enough,
  $S\m(X_\eps)$ and $S\m(Y_\eps)$ are contained in the $\eps$-neighbourhood of $X'$ and $Y'$ respectively.  Denote by
  $X'_\eps,Y'_\eps$ these neighbourhoods.  As above, 
since $X'$ and $Y'$ are finite, 
there exists $p'$ and $\eps$ such that $R^{p'}(Y'\cap\calc)$ does not
  intersect $X'$, so $R^{p'}(Y'_\eps\cap\calc)$ does not intersect $X'_\eps$ for $\eps$ small enough.
This proves the lemma.
\end{proof}

We are ready to conclude the proof of Theorem \ref{thm_no_MR_in_free}
by exhibiting a non-trivial relation between $R$ and $R'$.
In view of Lemma \ref{lem_disj}, consider $p,p'\geq 1$ and $\eps>0$ such that $R'^{p'}R^p(X_\eps)\cap X_\eps=\es$.
By Proposition \ref{prop_support}, there exists $n\geq 1$ such that the support of $U=[[R',R^n],R^n]$ is contained in
$X_\eps$ (see Remark \ref{rem_bord}). 
Since $U$ and its conjugate under  $R'^{p'}R^p$ have disjoint supports, they commute.
In other words, we have the relation $[U,R'^{p'}R^p U R^{-p}R'^{-p'}]=\id$.

One easily checks that this relator is non-trivial, 
\ie that $[U,R'^{p'}R^p U R^{-p}R'^{-p'}]$, 
  viewed as an element of the free group $F$ freely generated by $R,R'$, is non-trivial in $F$.
First $[R',R^n]$ is  non-trivial in $F$. Since commuting elements of $F$ are powers of a common element \cite[Prop. I.2.17-18]{LyndonSchupp},
we see that $R^n$ does not commute with $[R',R^n]$, so $U\neq 1$.
The normalizer of the maximal cyclic subgroup containing $U$ in $F$ being cyclic 
 \cite[Prop. I.2.19]{LyndonSchupp},
we get that  $[U,R'^{p'}R^p U R^{-p}R'^{-p'}]=1$ would imply $[U,R'^{p'}R^p]=1$ in $F$,
which is impossible since $U$ and $R'^{p'}R^p$ are not powers of a common element.
Theorem \ref{thm_no_MR_in_free} is proved.
\end{proof}

\section{Application to Lie groups}
\label{sec_Lie}
Consider a connected Lie group $L$.

Recall  that the exponential map of a Lie group is a local diffeomorphism from the Lie algebra of the group, in the group. 
Any element in its image belongs to a one-parameter subgroup, and therefore admits roots of arbitrary order (in other words, it is infinitely divisible). See for instance  \cite[Chap.2, Proposition 3.2]{OnishschikVinberg}.

Recall also that there exists a unique maximal
connected solvable normal Lie subgroup, ${\rm Rad}L$, called the
radical of $L$, and $L/{\rm Rad}L$ is semisimple (see \cite[Chap. 2,
Theorem 5.11]{OnishschikVinberg} for instance). On the other hand, every non-trivial
semisimple Lie  group contains a non-abelian free subgroup that is
dense \cite[final Corollary]{Kuranishi}.  In particular there are free
subgroups of $L$ generated by elements arbitrarily close to the identity.
\footnote{Note that this argument using \cite{Kuranishi} can be replaced by the
  genericity of free subgroups in $L$, provided there is one free
  subgroup to begin with, as recalled in the introduction}

\begin{prop}\label{prop;L_is_sol} Let $L$ be a connected Lie group. 
  If $L \hookrightarrow \IET$ is an injective homomorphism
(not necessarily continuous),  then 
 $L$ is solvable. 
\end{prop}

 \begin{proof}
Assume that $L$ is not solvable and that $L$ embeds in $\IET$. 
Then $L$ contains a non-abelian free group generated
by elements $a,b$ in the image of the exponential map, hence divisible.
By Corollary \ref{cor_growth} (\ref{item_vmr}), (\ref{item_divis}), the image of $a$ and $b$ in $\IET$ are conjugate to virtual multi-rotations.
This contradicts Theorem \ref{thm_no_MR_in_free}.
 \end{proof}

\begin{prop} \label{prop;lift}

Let  $\calA$ be an abelian connected Lie group and $\calB\simeq \mathbb{R}$, or  $\calB\simeq {\Bbb S}^1$.
 Let  $L$ be a Lie group that is an extension (of Lie groups) $1 \to \calA \to L \to \calB \to 1$ in which $\calA$ is closed.

Assume that   $L \hookrightarrow \IET$. 
 Then $L$ is abelian.
\end{prop}

\begin{rem}
  We don't assume that the embedding in $\IET$ is continuous. 
However, the maps occurring in the extension are continuous, and $L\to\calB$ is a surjective submersion.
\end{rem}

\begin{proof}
  We identify the elements of $L$ with their images in $\IET$.
Density arguments below take place in the Lie group.

  Write $\calA \simeq \mathbb{R}^n \times ({\Bbb S}^1)^m$, and choose   a finite subset $F\subset \calA$ 
  generating a dense subgroup.
  Note that by Corollary \ref{cor_growth} (\ref{item_divis}), every element of $\calA$ is
  conjugate  to a virtual multi-rotation.
  Replacing every element of $F$ by a power, $\grp{F}$ remains dense in $\calA$.
Thus, we can assume that elements of $F$ are conjugate to irrational multi-rotations.
 
  First we claim that there exist a pair of commuting elements $Q_1, Q_2 \in L$
  both conjugate to irrational multi-rotations,  whose image 
  generates a dense subgroup in $\calB$ (if $\calB\simeq \mathbb{S}^1$,
  we only need one element, but if  $\calB\simeq {\Bbb R} $ we
  clearly need a pair).
Indeed, consider two elements $b_1,b_2\in \calB$ of infinite order, generating a dense subgroup of $\calB$.
One can choose $b_1,b_2$ close enough to the identity so that they have preimages $Q_1,Q_2\in L$ in the image of the exponential map.
One can choose $Q_1,Q_2$ in the same one-parameter subgroup, so $Q_1$ commutes with $Q_2$.
By Corollary \ref{cor_growth} (\ref{item_vmr}), (\ref{item_divis}),  $Q_1,Q_2$ are conjugate to virtual multi-rotations.
Replacing them by powers, we can assume that $Q_1,Q_2$ are conjugate to irrational multi-rotations.

Since $\calA$ is normal and abelian, for each $T\in F$ and each $Q_i$, $Q_i^n T Q_i^{-n}$ commutes 
with $T$ for all $n>0$. 
By Lemma \ref{lem_csa}, $T$ commutes with $Q_i$.
Since $\grp{F}$ is dense in $\calA$, $Q_i$ commutes with $\calA$.
Since $Q_1$ commutes with $Q_2$, $\grp{Q_1,Q_2,\calA}$ is abelian.
 Since $\grp{Q_1,Q_2,\calA}$ is dense in $L$, $L$ is abelian.

\end{proof}

\begin{SauveCompteurs}{Lie}
  \begin{theo} 
    \label{theo;Lie}
    If $L$ is a connected Lie group, and $L \hookrightarrow \IET$ is
    an injective homomorphism, 
    then $L$ is abelian.
  \end{theo}
\end{SauveCompteurs}

\begin{proof}  By Proposition
  \ref{prop;L_is_sol}, $L$ is solvable. Assume it is not abelian,
  and write $\{1\}= L_{n+1} \triangleleft L_n \triangleleft L_{n-1}
  \triangleleft \dots \triangleleft L_1 =L$ where $L_{i+1} =
  [L_i,L_i]$. One can assume $n$ to be minimal, that is $L_{n-1} $
  non-abelian.
  
  By induction, we see that each $L_i$ is connected. In particular, the closures $\ol{L_n}$
  and $\ol{L_{n-1}}$ are connected Lie subgroups of $L$.
  Since $L_n$ is abelian, so is $\ol{L_n}$.
  Moreover, since $[ L_{n-1}, L_{n-1}]\subset L_n$, $[\ol{L_{n-1}},\ol{L_{n-1}}]\subset\ol L_n$.

 This means that
  $\overline{L_{n-1}}/\overline{L_{n}}$ is a connected abelian Lie group, hence isomorphic to
  $ \mathbb{R}^k \times ({\Bbb S}^1)^j$.
Hence $\ol L_{n-1}$ is an extension (of Lie groups) $1\ra \ol{L_n}\ra \ol {L_{n-1}}
\stackrel{\pi}{\ra}  \mathbb{R}^k \times ({\Bbb S}^1)^j\ra 1$.

Let us write $
\mathbb{R}^k \times ({\Bbb S}^1)^j\simeq \bigoplus_{i\leq k+j}
\calB_i$, coordinate by coordinate. 

Since  $\ol{L_{n-1}}$ itself
is non-abelian,  there exists a least  index $i_0$ such that
$\pi^{-1}( \bigoplus_{i\leq i_0}\calB_i)$ is not abelian. Note that
$\pi^{-1}(\{1\}) = \ol{L_n}$ is abelian, so that $i_0\geq 1$.

 Consider $\pi^{-1}( \bigoplus_{i\leq i_0-1}\calB_i)$ (if $i_0=1$ this
 is $\ol{L_n}$). It is abelian, by definition of $i_0$, normal and closed in  $\pi^{-1}( \bigoplus_{i\leq
  i_0}\calB_i)$, as it is the kernel of the natural projection $\pi^{-1}( \bigoplus_{i\leq
  i_0}\calB_i) \to \calB_{i_0} $.  By Proposition \ref{prop;lift},
$\pi^{-1}( \bigoplus_{i\leq i_0}\calB_i)$  does not embed in $\IET$, and neither does $L$.

\end{proof}

\section{Generic subgroups are not free}
\label{sec_generic}

The main result of this section says that under some irreducibility
assumption, a ``generic`` subgroup of $\IET$ is not free. 
This genericity is defined in topological terms
using the natural topology on $\IET$ where varying the lengths of the intervals of continuitly of an IET with fixed combinatorics describes a simplex
(see below or \cite[chap. 14.5]{HaKa_introduction} for a formal definition).
Note however that group operations are not continuous in this topology.

More precisely, 
if $I_1,\dots, I_n$ (resp. $J_1,\dots,J_n$) are the continuity intervals of $T$ (resp. $T\m$), ordered in an increasing fashion,
then the underlying permutation $\sigma$ is the permutation of $\{1,\dots,n\}$ so that $T(I_i)=J_{\sigma(i)}$.
We denote by $\IET_\sigma$ the set of all interval exchange transformations (with $n-1$ points of discontinuity)
whose underlying permutation is $\sigma$. Assigning to a transformation $T\in\IET_\sigma$ the $n$-tuple of lengths of its continuity intervals,
yields a natural bijection of $\IET_\sigma$ with the open simplex $\Delta_{n-1}$ of dimension $n-1$.
Clearly, $\sigma$ occurs as the permutation underlying some IET if and only if 
for all $i<n$, $\sigma(i+1)\neq \sigma(i)+1$, and we denote by $\frakS_n$ the set of such permutations.
Thus $\IET=\dunion_{n\geq 1} \dunion_{\sigma\in\frakS_n} IET_\sigma$ 
and the topology we consider is the one for which each subset $\IET_\sigma$ is open and the bijection above with $\Delta_{n-1}$ is a homeomorphism.

 \begin{dfn}\label{dfn_admissible}
The permutation $\sigma$ of $\{1, \dots, n\}$
 is \emph{admissible} if there is no $m \in \{1, \dots, n\}$ such that $\sigma(m)=m$  and $\sigma$ preserves $\{1,\dots,m-1\}$ 
(and therefore $\{m+1, \dots, n\}$).  
We denote by $\IET_a\subset \IET$ the set of all transformations whose underlying permutation is admissible.
 \end{dfn}

The main result of this section is Theorem \ref{theo_gnf}:
 \begin{SauveCompteurs}{gnf}
   \begin{theo}\label{theo_gnf}
There is a dense open set $\Omega_a \subset \IET \times \IET_a$   such that for all $(S,T) \in \Omega_a$,
 $\langle S,T\rangle$ is not free. 
\end{theo}
 \end{SauveCompteurs}

\begin{remark}
 In the statement, $S$ is not required to be in $\IET_a$.
\end{remark}

\subsection{Explanation in a simpler situation}
\label{sec_circle}
Before going into more technical arguments, we sketch a proof of the analogous result in a simpler situation.
Here, we consider interval exchange transformations on the circle $\bbR/\bbZ$.
Denote by $R_\theta$ the rotation $x\mapsto x+\theta$ on $\bbR/Z$. 
There is a natural topology on $\IET(\bbR/\bbZ)$ such that $R_\theta\circ T$ varies continuously with $\theta$.
This possibility to \emph{drift} $T$ will be important below.
For instance, given a permutation $\sigma$ of $\{1,\dots,n\}$ 
there is a natural set of IETs, 
parametrized by $\bbR/\bbZ\times \bbR/\bbZ\times \Delta_{n-1}$.
The two factors $\bbR/\bbZ$ indicate the initial points of $I_1$ and $J_1$,
and the simplex $\Delta_{n-1}$ defines the lengths of the intervals (note that in this setting, the permutation $\sigma$
is not uniquely defined in terms of $T$, but only its double coset modulo the cycle $(1,\dots, n)$).

Consider $S_0,T_0$ two IETs on $\bbR/\bbZ$ whose points of discontinuity 
are in the finite set $D_q=(\frac{1}{q}\bbZ)/\bbZ\subset \bbR/\bbZ$, and whose translation lengths are in $D_q$. 
Then $D_q$ is invariant under $S_0$, and $S_0^{q!}$ is the identity. 
If $S$ is a small perturbation of $S_0$, then $S^{q!}$ agrees with a very small translation
on each interval that is contained in the complement of a small neighbourhood of $D_q$.
Now consider $T$ a small perturbation of $T_0$.
Since two translations commute, this implies that the support of the commutator 
$[S^{q!},TS^{q!}T\m]$ is contained in a small neighbourhood
of $D_q$, say $N_\eps(D_q)$ (Lemma \ref{lem;commut1.5}).

Given $\theta>0$, define $T_\theta=R_\theta \circ T_0$, and consider $T$ a small perturbation of $T_\theta$.
If $\theta$ is small enough, and if $T$ is close enough to $T_\theta$, 
the discussion above applies in particular to $T$.
Note that all translation lengths of $T_\theta$ are equal to $\theta$ modulo $\frac{1}{q}\bbZ$.
We view $\theta$ as a drift since modulo $\frac{1}{q}\bbZ$, $R_\theta\circ T_0$ moves points uniformly by $\theta$ to the right.
Because of this drift, some power of $R_\theta\circ T_0$ sends $N_\eps(D_q)$ disjoint from itself.
If we perturb $T_0$ into a transformation $T$, all translation lengths of $T$ are very close to $\theta$ modulo $\frac{1}{q}\bbZ$
so $T$ still features this drift.
Thus, if $T$ is close enough to $R_\theta\circ T_0$, some power of $T$ will also send $N_\eps(D_q)$ disjoint from itself.
This means that for some $k\neq 0$, the support of $[S^{q!},TS^{q!}T\m]$ will be disjoint from its image under $T^k$,
so $[S^{q!},TS^{q!}T\m]$ commutes with its conjugate by $T^k$.
This prevents $\grp{S,T}$ from being free.

In what follows, we run the same strategy, except that the drift is not given by composition by a rotation,
but by a suitably designed modification of the IET $T_0$.

\subsection{Linear maps}

Now we come back to $\IET([0,1))$.
Given $\sigma\in\frakS_n$, one has a homeomorphism $\lambda:IET_\sigma\ra \Delta_{n-1}$ assigning
to an IET $T$ the tuple $\lambda(T)=(l_1,\dots,l_n)$ of lengths of its intervals of continuity.
This allows to define the $\ell^1$-metric on $\IET_\sigma$ by $d(T,T')=||\lambda(T)-\lambda(T')||_1=\sum_{i=1}^n |l_i-l'_i|$.
We extend this definition to $\IET$ by saying that two transformations with distinct underlying permutations 
are at infinite distance.

The tuple $\beta(T)=(b_1<\dots <b_{n-1})$ of points of discontinuity of $T$ 
is related to $(l_1,\dots,l_n)$ 
by the formulae
$l_i = b_i-b_{i-1}$ (with the obvious conventions $b_0=0, b_n=1$), 
and $b_i = \sum_{j\leq i} l_i$. 
This implies that for all $i$, $|b_i-b'_i|\leq d(T,T')$.

We will also be interested in the tuple of translation lengths of an IET:
given a transformation $T \in \IET_\sigma$, 
define $\phi(T)=(t_1,\dots,t_n)$ where $t_i$ is the translation length of $T$ in restriction to its $i$-th continuity interval.
   The following is immediate.

     \begin{lemma}\label{lem;linearity_Phi}
       The map $\phi\circ \lambda\m:(l_1,\dots,l_n)\mapsto (t_1,\dots,t_n)$ is linear.
 Explicitly:
     $$\begin{array}{ccc}
       t_i      & = &  -\sum_{j=1}^{i-1} l_j + \sum_{j=1}^{\sigma(i)-1} l_{\sigma^{-1}(j)}
   \end{array} $$
In particular, for all $i\in\{1,\dots,n\}$, $|t_i-t'_i|\leq 2 d(T,T')$.
     \end{lemma}

\subsection{$q$-rationality, and small support}

  A transformation $S_0\in \IET([0,1))$ is called \emph{$q$-rational}, for
  $q\in \mathbb{N}$,  if all
  its discontinuity points are in $\frac1q \mathbb{N}$. Clearly, in
  such case, all its translation lengths are in $\frac{1}{q}\bbZ$, and $S_0^{q!} = \id$.
  We also note that a $q$-rational transformation has at most 
$q$  continuity intervals.

  \begin{lemma} For all $\epsilon >0, q\in \mathbb{N}, m\in \mathbb{N}$, there exists
    $\eta>0$ such that, if $S_0$ and $T_0$ are $q$-rational, if $w(s,t)$ is a
    word of length $\leq m$ in the letters $s^{\pm 1}, t^{\pm 1}$ such that $w(S_0,T_0)
    = \id$, and if $S$
    and $T$ are $\eta$-close to $S_0$ and $T_0$ respectively, then on each interval of $[0,1)\setminus
    \calN_\epsilon (\frac1q \mathbb{N})$, the transformation $w(S,T)$  induces a translation of length $<\epsilon$.      
  \end{lemma}

\begin{proof}
Let $\epsilon >0$, and choose $\eta < \frac{\epsilon}{2m} $.
Let $I=(a,b)$ be any connected component of $[0,1) \setminus \frac1q \mathbb{N}$. 
Points of discontinuity of $S$ (resp. $T$) are within $\eta$ from those of $S_0$ (resp. $T_0$).
Moreover, in restriction to $(a+2\eta, b-2\eta)$, $S$ and $S_0$ (resp. $T$ and $T_0$) are translations of amplitude differing by at most $2\eta$ 
by Lemma \ref{lem;linearity_Phi}. 

It follows that, if $k\geq 1$, the endpoints of $S( (a+2k\eta, b-2k\eta))$ avoid the $2(k-1)\eta$-neighborhood of  
$\frac1q \mathbb{N}$ (and similarily for $T$). This is enough to ensure, by induction,  that, on each component of  $[0,1)\setminus
    \calN_{2m\eta} (\frac1q \mathbb{N})$ the transformation $w(S,T)$ is a translation, and its amplitude differs from 
that of $w(S_0,T_0)$ of at most $2m\eta$. Since  $w(S_0,T_0)$ is the identity, the result follows.
\end{proof}

  \begin{lemma}\label{lem;small_supp_comm}
     For all $\epsilon >0, q\in \mathbb{N}$, there exists
    $\eta>0$ such that, if $S_0$ and $T_0$ are $q$-rational, and if $S$
    and $T$ are $\eta$-close to $S_0$ and $T_0$ respectively,  then on each interval of $[0,1)\setminus
    \calN_\epsilon (\frac1q \mathbb{N})$, the transformation
    $[S^{q!},TS^{q!}T^{-1}]$ induces the identity.
  \end{lemma}

  \begin{proof}
    Apply the previous lemma to the words  $s^{q!}$ and $ts^{q!}t^{-1}$ for $\epsilon/4$.
Then apply Lemma \ref{lem;commut1.5}. 
  \end{proof}

   \subsection{Drifting the support}

     \subsubsection{Drift vector}\label{subsubsec_drift}

In this section, we introduce the notion of \emph{drift vector}, a substitute to the composition by a rotation that can be used
over a circle.

     Fix a permutation $\sigma$, and consider the linear map $\Phi_\sigma=\phi\circ\lambda\m$
that assigns to tuple of lengths $(l_i)_{i=1,\dots,n} $ the tuple of translation lengths $(t_i)_{i=1,\dots,n} $
  (see Lemma \ref{lem;linearity_Phi}).
     \newcommand{\dr}{\overrightarrow{dr}}
     \newcommand{\dl}{\overrightarrow{dl}}
     
     \begin{defi}\label{defi;moveable}
       We say that a permutation $\sigma\in\frakS_n$ is \emph{driftable}  if there exists
       a vector $\dl=(dl_1,\dots,dl_n)\in \bbR^n$ with $\sum  dl_i =0$ such that the vector  $\dr=\Phi_\sigma(\dl)\in\bbR^n$
       has only positive (non-zero) coordinates.

       When it exists, we call $\dr$ a \emph{drift vector}, and $\dl$ a \emph{drifting direction}.
     \end{defi}
     
     The vector $\dr=\Phi_\sigma(\dl)$  represents the change of translation lengths induced by the change
     of lengths of intervals $\dl$. The point of a drift vector is therefore that all translation lengths increase
     when we change the lengths of the intervals by a positive multiple of $\dl$.
     In Proposition \ref{prop;admis_is_moveable}, we will see that $\sigma$ is driftable if and only if it is admissible.

     In all the rest of this section, we assume that the permutation $\tau$ is driftable, and we fix $\dl$ and $\dr$ as above.

     Consider $T_0\in\IET_\tau$ a given $q$-rational transformation.
    Given $\theta>0$, define $T_\theta\in IET_\tau$ in terms of lengths of its continuity intervals
     by $\lambda(T_\theta)=\lambda(T_0)+\theta\dl$. 
       
\newcommand{\drmax}{dr_{\mathrm{max}}}
\newcommand{\drmin}{dr_{\mathrm{min}}}

     \begin{defi}
       We denote by  $\drmax=\max_{i=1}^n dr_i$ and  $\drmin=\min_{i=1}^n dr_i$ the \emph{maximal} and \emph{minimal drift} of $\dr$.
     \end{defi}

     \begin{lemma}\label{lem;precise_drift}
       All translation lengths of $T_\theta$ are in $[\theta \drmin,\theta \drmax]$ mod $\frac1q$.

       Moreover, if $||\lambda(T)-\lambda(T_\theta)||_1\leq \mu$,
then all
       translation lengths of $T$ are in $[\theta \drmin-2\mu, \theta \drmax+2\mu]$ mod $\frac1q$.
     \end{lemma}

     \begin{proof}
       All translation lengths of $T_0$ are $0$ mod
       $\frac1q$. By linearity of $\Phi_\tau$ (as defined at the beginning of Section \ref{subsubsec_drift}, the translation vector
       of $T_\theta$ is  $\Phi_\tau(\lambda(T_0)+\theta\dl) =
       \Phi_\tau(\lambda(T_0))+\theta\Phi_\tau(\dl)=\phi(T_0)+\theta \dr$. Therefore, the translation lengths of
       $T_\theta$ are the coordinates of $\theta\dr$ mod $\frac1q$. The first assertion
       follows.

       For the second statement, express $\lambda(T)$ as $\lambda(T) = \lambda(T_0) +
       \theta\dl +\vec\epsilon$, where $||\vec \epsilon||_1 \leq \mu$.
       By linearity, $\phi(T)=\phi(T_0)+\theta\dr+\Phi_\tau(\vec\eps)$, 
       which, mod $\frac1q$, gives $\theta\dr+\Phi_\tau(\vec\eps)$.
       Now, writing $\vec\eps=(\eps_i)_{i=1}^n$, we have by Lemma \ref{lem;linearity_Phi},
       $$\Phi_\tau (\eps) = \left( -  \sum_{j=1}^{i-1} \epsilon_j + \sum_{j=1}^{\tau(i)-1} \epsilon_{\tau^{-1}(j)} \right)_{i=1,\dots,n}.$$ 
      Since  $|| \vec\epsilon||_1 \leq \mu$, every coordinate of this vector is at most $2\mu$.
     \end{proof}

     \subsubsection{Generic absence of free groups}

     \begin{prop}\label{prop;small_open_sets}
       Assume that $S_0,T_0$ are $q$-rational, and that the permutation $\tau$ underlying $T_0$ is driftable.

       Then, there exist a neighbourhood $\calU$ of $S_0$ 
and an open set $\calV$
       that accumulates on $T_0$, 
such that whenever $(S,T)\in \calU\times \calV$, 
       $\grp{S,T}$ is not  free of rank $2$.
     \end{prop}

     \begin{rem}
      We don't claim that $\calV$ is a neighbourhood of $T_0$, only that $T_0\in\ol \calV$.
     \end{rem}

  \begin{proof} 
    Let $\dl,\dr$ be a drifting direction and a drifting vector as above,
    and $\rho=\frac{\drmax}{\drmin}$.
     Let $\epsilon <  \min(\frac{1}{100\rho},\frac{1}{10q})$, and let $\eta$ be given by Lemma
     \ref{lem;small_supp_comm}. 
Let $\calU$ be the set of transformations $S$ at distance $\eta$ from $S_0$.

     Now consider $\theta>0$  such that: $\theta \drmax\leq \eta/10$, $\theta ||\dl||_1\leq \eta/10$,
     and $\theta \drmin<\eps$.
     The transformation $T_\theta$ defined above is then at distance $\leq \eta/10$ from $T_0$.
     Finally, let $\mu <\frac{\theta\drmin}{4}$, $\mu<\eta/2$.  
Let $\calV_\theta$ be the set of tranformations $T\in\IET_\tau$
at distance $<\mu$ from $T_\theta$. Note that any $T\in\calV_\theta$ is at distance  $<\eta$ from $T_0$. 
In particular, any $(S,T)\in \calU\times\calV_\theta$ 
     satisfies the assumptions of Lemma \ref{lem;small_supp_comm}. It follows that $U= [S^{q!},TS^{q!}T^{-1}]$ has support in the $\epsilon$-neighborhood of
     $\frac1q \mathbb{N}$.

     We claim that there exists $k$ such that $T^kUT^{-k}$ has support
     disjoint from that of $U$. The claim
     implies that $[T^kUT^{-k}, U]=\id$. As in the end of the proof of Theorem \ref{thm_no_MR_in_free}, 
     one checks that this word is non-trivial (as an element of the abstract free group on $S,T$), 
     so $\grp{S,T}$ is not free group of rank $2$.
Taking for $\calV$ the union of all $\calV_\theta$ for $\theta$ small enough concludes the Proposition.

     Let us now prove the claim. 
     It suffices to show that there exists
     $k$ such that all translation lengths of $T^k$ lies in $[2\eps,\frac1q-2\eps]$ modulo $\frac1q \mathbb{Z}$.

     We know from Lemma \ref{lem;precise_drift} that all translation
     lengths of $T$ are in $[\theta\drmin-2\mu, \theta\drmax +2\mu
     ]$ mod $\frac1q$. According to our choice of $\mu$ they are therefore in
     $[\theta\drmin/2, 2\theta\drmax]$ mod $\frac1q$. Therefore, any translation length $\tau$ of $T^k$
     satisfies $\tau\in[k\theta\drmin/2, 2k\theta\drmax]+\frac1q\bbZ$ . 

     Since $\theta\drmin<\epsilon$, there exists $k$ such that
     $2\epsilon < k\theta\drmin/2 \leq 3\epsilon$. For such a $k$, one has
     $2k\theta\drmax=2k\theta\drmin\times \rho \leq 12\epsilon\rho$.
     Since $\epsilon < \min(\frac{1}{100\rho q},\frac{1}{10q})$ , one successively gets  $12\epsilon\rho\leq \frac{12}{100q} 
     <\frac1q-2\eps$.
     This establishes the claim.
  \end{proof}

  \subsubsection{Driftable permutation}

To conclude this section, we prove that a permutation is driftable if and only if it is admissible.
Recall that $\sigma\in \frakS_n$ is driftable if there exists a change of lengths $\dl=(dl_1,\dots,dl_n)$
with $\sum dl_i=0$, that increases all translation lengths of the corresponding IET (Definition \ref{defi;moveable}),
and $\sigma$ is non-admissible if there is an $m\in\{1,\dots,n\}$ 
such that $\sigma(m)=m$ and $\sigma(\{1,\dots,m-1\})=\{1,\dots,m-1\}$ (Definition \ref{dfn_admissible}).

\begin{prop}\label{prop;admis_is_moveable}
 Let $\sigma \in \frakS_n$. Then $\sigma$ is driftable  if and only if it is admissible.
\end{prop}

Let us note that, together with Proposition \ref{prop;small_open_sets}, 
 this immediately implies Theorem \ref{theo_gnf}.

Before proving  proposition \ref{prop;admis_is_moveable}, we treat a particular case
related to the example on the circle described in Section \ref{sec_circle}. Start from an IET $T$ on the circle with $n$ intervals of continuity, 
and cut the circle at a discontinuity point of $T$ that is not a discontinuity point of $T\m$ (assuming that there is such a point).
Then we get an IET $T_0$ on $[0,1)$. Looking at the interval of continuity of $T\m$ that has been cut, we see that the underlying permutation 
$\tau\in\frakS_{n+1}$ of $T_0$ satisfies $\tau\m(1)=\tau\m(n+1)+1$.
In other words, $\tau(i_0) = n+1$ and $\tau(i_0+1) = 1$ for some $i_0\in\{1,\dots,n\}$. 
Then define  $T_\theta\in \IET_\tau$ to be the transformation with same underlying permutation,
and same lengths of continuity intervals $l'_i=l_i$ for $i\neq i_0,i_0+1$,
whereas $l'_{i_0}=l_{i_0}-\theta$ and $l'_{i_0+1}=l_{i_0+1}-\theta$
for some $\theta$ small enough.
Denote by $t_i$ (resp. $t'_i$) the translation lengths of $T$ (resp $T_\theta$).
Then one easily checks that $t'_i=t_i+\theta$ for all $i=1,\dots,n+1$ (we leave the proof of this fact to the reader).
This shows that $\tau$ is driftable.

\begin{proof}[Proof of proposition \ref{prop;admis_is_moveable}.]
Recall that the map $\Phi_\sigma$ is the map assigning translation lengths to interval lengths defined before
Definition \ref{defi;moveable}.
Assume that $\sigma$ is not admissible, and consider $m$ such that $\sigma(m)= m$ and $\{1, \dots
,m-1\} $ is $\sigma$-invariant. 
Clearly, whatever the lengths of the continuity intervals, the corresponding interval exchange $T$
is the identity on $I_m$ (this follows from Lemma \ref{lem;linearity_Phi}). This implies that for all
$\dl=(dl_i)_{i\leq n}$ such that $\sum dl_i =0$, the corresponding coordinate of $\Phi_\sigma(\dl)$ vanishes.

Now let us assume that $\sigma$ is admissible. 
Say that a pair $\{i_1,i_2\}$ is inverted if $i_1<i_2$ and $\sigma(i_1)>\sigma(i_2)$.
The fact that $\sigma$ is admissible implies that each $i\in\{1,\dots,n\}$ lies in an inverted pair
since otherwise, $\{1,\dots,i-1\}$ and $\{i+1,\dots, n\}$ are $\sigma$-invariant.
Given an inverted pair $\{i_1,i_2\}$,
consider the vector $\dl=(dl_1,\dots,dl_n)$ 
where $dl_{i_1}=-1$, $dl_{i_2}=+1$, and $dl_i=0$ otherwise.
Consider $\dr=\Phi_\sigma(\dl)$ and write $\dr = (t_k)_{k=1\dots n}$ with $t_k = -\sum_{j=1}^{k-1} dl_j + 
\sum_{j=1}^{\sigma(k)-1} dl_{\sigma^{-1}(j)} $. 
The contribution of the first sum is $1$ for $i_1< k \leq i_2$ and $0$ otherwise.
The contribution of the second sum is $1$ for $\sigma(i_2)<\sigma(k)\leq \sigma(i_1)$ and $0$ otherwise.
It follows that $t_k\geq 0$ for all $k$, and that $t_{i_2}$ and $t_{i_1}$ are at least $1$.

Since any $i\in\{1,\dots,n\}$ lies in an inverted pair, adding all vectors $\dl$ corresponding to all inverted pairs yields a vector 
whose image under $\Phi_\sigma$ has only positive entries.
This proves that $\sigma$ is driftable.
\end{proof}

\section{No infinite Kazhdan groups in IET}\label{sec_lattice}

Using non-distortion of cyclic subgroups in $\IET$,
Novak has proved that non-uniform lattices in higher rank 
semi-simple Lie groups with finite center have finite image in $\IET$ \cite{Novak_discontinuity}.
The following applies to lattices (including cocompact ones)
in Kazhdan semi-simple Lie groups.

\begin{thm}\label{lattice}
Let $G$ be a finitely generated Kazhdan group.
Then any morphism from $G$ to $\IET$ has finite image.
\end{thm}

Here, we give an argument based on the fact that orbits have polynomial growth.

\begin{lemma}
Let $G <\IET$ be a finitely generated group.
Then every orbit of $G$ for its action on $[0,1)$
has polynomial growth.
\end{lemma}

\proof
Let $g_1, \cdots, g_n \in \IET$ be a finite set generating $G$ as a semigroup.
For each $i$, let $t_i^j \in [0,1), (1 \le j \le m_i)$ be the 
 translation lengths of $g_i$. 
Set $M=m_1 + \cdots +m_n$ the total number of translations involved.
For $R \in {\Bbb N}$, let $B_R \subset G$ be 
the set of all elements of word-length at most $R$.
For any $x \in [0,1)$, commutation of translations implies that 
$$|B_R.x| \le R^M.$$
\qed

Recall that a group $G$ is \emph{Kazhdan} (or has Kazhdan property $(T)$) 
if any unitary representation of $G$
having almost fixed vectors has a fixed vector. 
Witte-Morris pointed out to us the following result:

\begin{lem}[see for instance Theorem B in \cite{Stuck_growth}]
%
 Let $G$ be a finitely generated Kazhdan group, and $G\actson X$ be an action on a set.
If the orbit of some $x\in X$ is infinite, then it has exponential growth.
\end{lem}

This implies that all orbits of $G$ are finite.
We conclude the proof of Theorem \ref{lattice} by the following fact.

\begin{lem}
  Let $G<\IET$ be a finitely generated group. Assume that every orbit of $G$ is finite.
Then $G$ is finite.
\end{lem}

\begin{proof}
We view work in $\IET(\bbR/\bbZ)$.
Let $S$ be a finite generating set of $G$.
Let $D\subset\bbR/\bbZ$ be the union of all orbits of the discontinuity points of the elements of $S^{\pm1}$.
The set $D$ is finite and we can assume $D\neq \es$ since otherwise the result is clear.
All generators are continous on the $G$-invariant set 
$\bbR/\bbZ\setminus D$. It easily follows that on each connected component 
of $\bbR/\bbZ\setminus D$, every orbit has the same finite cardinality.
It follows that $G$ is finite.
\end{proof}






\section{Residual finiteness}\label{sec_RF}

Let $S$ be a finite set. A group \emph{marked} by $S$ is a finitely generated
group with a quotient map $\bbF_S\to G$, where $\bbF_S$ is the free group
over $S$.

A sequence of marked groups $\pi_n: \bbF_S \to G_n$ converges to a marked group 
$\pi: \bbF_S\to G$ if for all $R>0$, if $B_R$ is the ball of radius $R$ at $1$ in $\bbF_S$, then $\ker (\pi_n ) \cap B_R = \ker (\pi)\cap B_R$
for sufficiently large $n$. 

The aim of this section is to obtain the following.
\begin{SauveCompteurs}{RF}
  \begin{theo}\label{thm_RF}
    Every finitely generated subgroup of $\IET$ is the limit of a
    sequence of finite groups in the space of marked groups.

    In particular, every finitely presented subgroup of $\IET$ is
    residually finite.
  \end{theo}
\end{SauveCompteurs}

The second statement immediately follows from the first since
 any finitely presented marked group $G$,
 has a neighborhood consisting of quotients of itself (i.e., for all $\pi':\bbF_S\to G'$ close enough to $\pi:\bbF_S\to G$,
 $\pi'=q \circ \pi$ for some $q:G \to G'$).
This is because if $R$ is the length of the longest defining 
relation of $G$, then all defining relations
are in $\ker(\pi) \cap B_R$, hence those relations are 
in $\ker(\pi_n)$ for sufficiently large $n$, so that $G_n$ is a quotient 
of $G$.

In particular Thompson's groups cannot embed in $\IET$.

\begin{cor}
Thompson's groups $F$ (on the interval) and $T,V$ (on the circle and the cantor set) 
 are not subgroups of $\IET$.
\end{cor}

\begin{proof} 
It is known that $F,T$ and $V$ are finitely presented
(cf. \cite{CFP_thompson}).
The groups $T$ and $V$ are simple and infinite, therefore not residually finite, 
and do not embed in IET by Theorem \ref{thm_RF}.

The group $F$ is not simple, but any proper quotient is abelian \cite{CFP_thompson}. 
Since $F$ is infinite and non-abelian, 
it is not residually finite: any commutator has to be mapped 
to the trivial element in a finite quotient.
Therefore, $F$ is not a subgroup of $\IET$.
\end{proof}

The following is also immediate from Theorem 
\ref{thm_RF}. We'd like to thank Alexei Muranov for 
turning our attention to the idea of sofic groups. See \cite{Gromov_endomorphisms},
or \cite{Pestov_hyperlinear} for definitions.
\begin{cor}\label{sofic}
The group $\IET$ is sofic.
\end{cor}
In view of Theorem \ref{thm_RF}, this follows immediately from the fact that finite groups are sofic, limits of sofic groups are sofic, 
and that soficity for a group holds
if and only if it holds for its finitely generated subgroups \cite{Cornulier_sofic}.

The idea of the proof of the theorem is to produce a finite group by replacing the coefficients (\ie the lengths) of the interval exchanges considered
by rational numbers, keeping the fact that a given finite collection of words represent the identity or not.
Then, we use the fact that the group generated by interval exchanges with rational coefficients is always finite.
The control on trivial and non-trivial elements is based on the fact that any system of linear equations and inequations with rational coefficients
that has a real solution, also has a rational solution. 

\begin{prop}[{\cite[Cor 3.1.17]{Marker_model}}]\label{prop;lin_sys}
 If a linear system of equations and inequations with rational coefficients has a real solution, then it has a rational solution. 
\end{prop}

In this statement, we allow equalities $f(x)=0$ and strict inequalities $f(x)>0$ for $f:\bbR^n\ra \bbR$ an affine map with rational coefficients.

\begin{proof}[Proof of Theorem \ref{thm_RF}]
 Define a rational \emph{polyhedron} in $\bbR^N$ as the set of solutions of
  finitely many equations and inequations of the form $f(x)=0$ or $f(x)>0$ where $f:\bbR^N\ra \bbR$ is an affine map.  A \emph{PL}
  subset of $\bbR^N$ is a finite union of rational polyhedra.
  If $C\subset \bbR^N$ is a PL subset, we say that $f:C\ra \bbR^M$ is \emph{PL} if there is a partition of $C$ into finitely many
  rational polyhedra in restriction to which $f$ coincides with an affine map $\bbR^N\ra \bbR^M$ (note that contrary to standard
  terminology, we don't require $f$ to be continuous).  Obviously, composition and inverses of PL maps are PL, and the preimage of
  a PL subset by a PL map is a PL subset.
The intersection of finitely many PL subsets is a PL subset. 
  By Proposition \ref{prop;lin_sys}, any non-empty rational polyhedron contains a rational point,
  so does any non-empty PL subset.

  Note that if $C_1\subset \bbR^{N_1}$ and $C_2\subset \bbR^{N_2}$ are PL subsets, then $C_1\times C_2\subset \bbR^{N_1+N_2}$ is a
  PL subset. Moreover, there is a PL embedding of $C_1\dunion C_2$ to a PL subset of $\bbR^{N_1}\times \bbR^{N_2}\times \bbR$ as
  $C_1\times \{0\}\times \{0\}\cup \{0\}\times C_2\times \{1\}$.

  Since for each permutation $\sigma\in \frakS_n$, $\IET_\sigma$ is parametrised by an open simplex of dimension $n-1$, 
  and embeds naturally in $\bbR^{n}$,
  the set $\IET_{\leq n}$ of interval exchange transformations of $[0,1)$ with at most $n$ points of discontinuity
  can be identified with a PL subset of some $\bbR^{N_n}$.

The following lemma is a key.
\begin{lem}
  The group law $\IET_{\leq n}\times \IET_{\leq m}\ra \IET_{\leq n+m}$ is a PL map.
\end{lem}

\begin{proof}
It is clear that the map $S\mapsto S\m$ is PL.
It is enough to prove that the group law is PL on $\IET_{\sigma}\times\IET_{\tau}$.
Consider $(S,T)\in \IET_{\sigma}\times\IET_{\tau}$.
Recall that $\Delta(S)$ denotes its set of discontinuity points (which we view as an ordered tuple $b_1<\dots<b_n$.
Clearly,  $\Delta(S)$ and  $T\Delta(T)$ are affine functions
of $(S,T)$.
Thus each possible combinatorics for $\Delta(S)\cup T\Delta(T)$ defines a rational polyhedron, and we can assume that one of these combinatorics is fixed.
Denote by $K_k$ the set of connected components of  $[0,1)\setminus \Delta(S)\cup T\Delta(T)$, and
note that $S\circ T$ is continuous on the intervals $T\m(K_k)$.
The endpoints of $K_k$ are affine functions of $(S,T)$.
Since the evaluation map $[0,1)\times \IET_{\leq m}\ra [0,1)$ sending 
$(x,T)$ to $T(x)$ is PL, the endpoints of $T\m(K_k)$ are PL (indeed affine) functions of $(S,T)$.
The amplitude of translation of $S\circ T$ on $T\m(K_k)$ is the sum of the amplitude of $T$ on
$T\m(K_k)$ and of the amplitude of $S$ on $K_k$.
It easily follows that the group law is PL.
\end{proof}

  Now consider a subgroup $G\subset \IET$ generated by $g_1,\dots,g_k\in \IET_{\leq n}$.  We see $G$ as a quotient of the free
  group $\bbF_k$ over $\{x_1,\dots, x_k\}$, in which $x_i$ is mapped on $g_i$.  
Denote by $B_R\subset {\Bbb F}_k$ the set of words of
  length $\leq R$.  We need to show that, for all $R>0$, there is a finite quotient $Q$ of $\bbF_k$ such that for any element
  $w(x_1,\dots, x_k)\in B_R$, $w(g_1,\dots, g_k)$ is trivial in $G$ if and only if the image of $w$ in $Q$ is trivial.

  For each $w\in B_R$, consider the PL map $f_w:(\IET_{\leq n})^k\ra \IET_{\leq Rn}$ defined by $(t_1,\dots,t_k)\mapsto
  w(t_1,\dots,t_k)$.  Consider $C_w\subset (\IET_{\leq n})^k$ the PL subset defined as $C_w=f_w\m(\{\id\})$ if $w(g_1,\dots,g_k)=\id$
  in $G$, and as $C_w=f_w\m(\IET_{\leq Rn}\setminus \{\id\})$ otherwise.  
  The intersection $C$ of all $C_w$ is a PL subset that
  contains $(g_1,\dots,g_k)$ by construction.  It follows that $C$ contains a point with rational coordinates $(g'_1,\dots,  g'_k)$.  
In particular, the subgroup $Q$ of $\IET$ generated by $g'_1,\dots,g'_k$ is finite.  Since $(g'_1,\dots,g'_k)\in
  C$ this finite quotient of $\bbF_k$ satisfies that for all $w\in B_R$, $w(g_1,\dots, g_k)$ is trivial in $G$ if and only if
  $w(g'_1,\dots,g'_k)$ is trivial in $Q$.
\end{proof}

\section{An example of subgroup}
\label{sec_example}

\begin{SauveCompteurs}{example}
  \begin{thm}\label{thm_example}
There is a subgroup $F\subset IET$ generated by two elements that contains
\begin{itemize}
\item an isomorphic copy of all finite groups
\item a free semigroup.
\end{itemize}
In particular this group is not linear and has exponential growth.
\end{thm}
\end{SauveCompteurs}

\begin{figure}[htbp]
  \centering
  \includegraphics{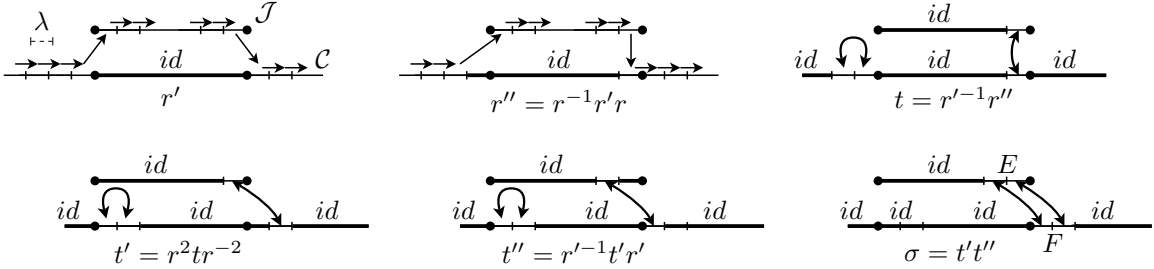}
  \caption{A non-linear example.}\label{fig_dyn}
\end{figure}

\begin{proof}
 Let $\calc=\bbR/2\bbZ$ be a circle of perimeter $2$, and $\calI=[0,1) \subset \calc$ an interval of length $1$.
Let $\calJ=[0,1)$ be a copy of $\calI$.
Consider the domain $\cald=\calC \cup \calJ$, and two iets $r,s$ on $\cald$ such that
$r$ is an irrational rotation on $\calc$ and the identity on $\calJ$, and $s$ is the involution that
switches $I$ and $J$ and is the identity on $C\setminus I$.

We prove that the subgroup $F=\langle r,s \rangle <IET$
contains every permutation group $\Sigma_n$.
First, up to replacing $r$ with some power, we may assume that the amplitude $\lambda$ of $r$
satisfies $0<\lambda< \frac1{10n}$.

We claim that $\grp{r,s}$ contains an involution $\sigma$ that exchanges
$E=[1-2\lambda,1)\subset \calJ$ and $F=[1,1+2\lambda)\subset \calc$,
 and that is the identity everywhere else.
The element $r'=srs$ can be thought as a rotation
of the circle $\calc '$ obtained from $\calc$
by replacing $\calI$ by $\calJ$.
Figure \ref{fig_dyn} shows the dynamics of $r'$, $r''=r\m r' r$, $t=r'^{-1} r''$,
$t'=r^2tr^{-2}$,  $t''=r'^{-1} t' r$, and finally $\sigma=t't''$.

Note that $r^2$ fixes $E$ and moves $F$ apart from itself.
Since $\lambda<<\frac{1}{n}$, the intervals $F$, $r^2(F)$,\dots $r^{2n}(F)$ are disjoint.
It follows that the group $\grp{\sigma, r^2\sigma r^{-2},\dots, r^{2n}\sigma r^{-2n}}$
permutes the intervals $E,F,\dots r^{2n}(F)$,
and is isomorphic to $\Sigma_{n+2}$.

We now prove that the semigroup of $F$ generated by $r$ and $r'=srs$ is a free semigroup.
Let $f:\cal D \to \cal C$ denotes the map which
is identity on $\cal C$ and maps $\cal J$ to $\cal I$ isometrically, and let $p,q$ be the initial and terminal endpoints of $\calI$.
Given a positive word $w$ on $\{r,r'\}$, denote by $|w|$ its word length and by $W$ the corresponding transformation.
Observe that for any $x\in \calD$, there exists $n\in\{0,\dots,|w|\}$
such that $f(W(x))=r^n(f(x))$. 
If $w\neq 1$, then for $x\in \calc\setminus\calI$, one has $n>0$, and since $r$ is irrational, this already implies $W\neq \id$.
We note for future use that a similar argument shows shows that $W(p)=p$ if and only if $W$ is a power of $r'$.
We still have to check that the mapping $w\mapsto W$ is injective.
 
Define 
\begin{eqnarray*}
D_k&=&f\m(\{p,q\}\cup r\m\{p,q\}\dots\cup r^{-k+1}\{p,q\}), \\
 \text{and } D'_k&=&f\m(\{p,q\}\cup r\{p,q\}\dots \cup r^{k-1}\{p,q\}).
\end{eqnarray*}

We see that for any positive word $w$, the discontinuity points of $W$ (resp. $W\m$) are in $D_{|w|}$ (resp. $D'_{|w|}$).
Since $r$ is irrational, $p\notin r\m(D_k)$, so $Wr$ is continuous at $p$.
We claim that $W'=Wr'$ is discontinuous at $p$. 
Indeed, if $W'$ was continuous at $p$, then since $r'=W\m W'$ is discontinuous at $p$, $W\m$ would be discontinous at $W'(p)$,
so $W'(p)\in D'_{|w|}$. Since $r$ is irrational and $p,q$ are not in the same orbit under $r$, 
this leaves only the possibility $W'(p)=p$ ($W'(p)=s(p)$ is impossible). It follows that $W'$ is a power of $r'$.
Since $p$ and $q$ are not in the same orbit under $r$, no proper power of $r'$ is continuous at $p$, so $W'=\id$, a contradiction,
which proves the claim.

We now prove that $W_1=W_2$ implies $w_1=w_2$ by induction on $n=\min(|w_1|,|w_2|)$.
For $n=0$, this is the fact that $w\neq 1\Rightarrow W\neq 1$.
For $n>0$, the continuity or discontinuity of $W_1=W_2$ at $p$ tells us that $w_1$ and $w_2$ start with the same letter,
and by induction, $w_1=w_2$.
This concludes the proof that $r,r'$ generate a free semigroup in $\IET$.

Finally, the fact that $F$ is not linear, follows immediately from the following fact.
\end{proof}

\begin{fact*}
Let $G$ be a finitely generated linear group. Then there exists $C$
such that the order of an element in $G$ of finite order is at most
$C$.
\end{fact*}

This fact is proved
in the proof of Theorem 36.2 (Schur) in \cite{CurtisReiner},
which says that any finitely generated periodic subgroup
of $GL(n,K)$ is finite if the characteristic of the field
$K$ is zero.
In the proof, it is shown that if $G < GL(n,K)$
is a finitely generated group then there exists
$m$ such that if $g \in G$ has finite order then $g^m=1$.
The essential part of the argument is to show that
a characteristic root of $g$ is the $m$-th root of
unity. For this part, it does not matter if the characteristic
of $K$ is zero or positive.
If the characteristic of $K$ is zero, it immediately follows
that $g^m=1$.
If the characteristic $p$ is positive,
then we use a Jordan normal form of $g$, write $g^m = I+N$ with $N$ nilpotent, and see that
$(g^{m})^{p^n}=1$.


\begin{thebibliography}{10}

\bibitem{CFP_thompson}
J.~W. Cannon, W.~J. Floyd, and W.~R. Parry.
\newblock Introductory notes on {R}ichard {T}hompson's groups.
\newblock {\em Enseign. Math. (2)}, 42(3-4):215--256, 1996.

\bibitem{Cornulier_sofic}
Yves Cornulier.
\newblock A sofic group away from amenable groups.
\newblock {\em Math. Ann.}, 350(2):269--275, 2011.

\bibitem{CurtisReiner}
Charles~W. Curtis and Irving Reiner.
\newblock {\em Representation theory of finite groups and associative
  algebras}.
\newblock AMS Chelsea Publishing, Providence, RI, 2006.
\newblock Reprint of the 1962 original.

\bibitem{Gromov_endomorphisms}
M.~Gromov.
\newblock Endomorphisms of symbolic algebraic varieties.
\newblock {\em J. Eur. Math. Soc. (JEMS)}, 1(2):109--197, 1999.

\bibitem{HaKa_introduction}
Anatole Katok and Boris Hasselblatt.
\newblock {\em Introduction to the modern theory of dynamical systems},
  volume~54 of {\em Encyclopedia of Mathematics and its Applications}.
\newblock Cambridge University Press, Cambridge, 1995.
\newblock With a supplementary chapter by Katok and Leonardo Mendoza.

\bibitem{Kuranishi}
Masatake Kuranishi.
\newblock On everywhere dense imbedding of free groups in {L}ie groups.
\newblock {\em Nagoya Math. J.}, 2:63--71, 1951.

\bibitem{LyndonSchupp}
Roger~C. Lyndon and Paul~E. Schupp.
\newblock {\em Combinatorial group theory}.
\newblock Classics in Mathematics. Springer-Verlag, Berlin, 2001.
\newblock Reprint of the 1977 edition.

\bibitem{Marker_model}
David Marker.
\newblock {\em Model theory, an introduction}, volume 217 of {\em Graduate
  Texts in Mathematics}.
\newblock Springer-Verlag, New York, 2002.

\bibitem{Novak_free}
Christopher Novak.
\newblock Interval exchanges that do not occur in free groups.
\newblock arXiv:1007.3940v1 [math.DS].

\bibitem{Novak_discontinuity}
Christopher~F. Novak.
\newblock Discontinuity-growth of interval-exchange maps.
\newblock {\em J. Mod. Dyn.}, 3(3):379--405, 2009.

\bibitem{Novak_continuous}
Christopher~F. Novak.
\newblock Continuous interval exchange actions.
\newblock {\em Algebr. Geom. Topol.}, 10(3):1609--1625, 2010.

\bibitem{OnishschikVinberg}
A.~L. Onishchik and {\`E}.~B. Vinberg.
\newblock Foundations of {L}ie theory [see {MR}0950861 (89m:22010)].
\newblock In {\em Lie groups and {L}ie algebras, {I}}, volume~20 of {\em
  Encyclopaedia Math. Sci.}, pages 1--94, 231--235. Springer, Berlin, 1993.

\bibitem{Pestov_hyperlinear}
Vladimir~G. Pestov.
\newblock Hyperlinear and sofic groups: a brief guide.
\newblock {\em Bull. Symbolic Logic}, 14(4):449--480, 2008.

\bibitem{Stuck_growth}
Garrett Stuck.
\newblock Growth of homogeneous spaces, density of discrete subgroups and
  {K}azhdan's property ({T}).
\newblock {\em Invent. Math.}, 109(3):505--517, 1992.

\end{thebibliography}


\begin{flushleft}
Fran\c cois Dahmani, Institut Fourier UMR 5582, Université de Grenoble\\
BP74, 38402 Saint Martin d'H\`eres cedex, France.\\
\texttt{francois.dahmani@ujf-grenoble.fr}\\[3mm]

Koji Fujiwara, Graduate School of Information Science, Tohoku University\\
Sendai, 980-8579, Japan.\\
\texttt{fujiwara@math.is.tohoku.ac.jp}\\[3mm]

Vincent Guirardel, Université de Rennes 1, IRMAR, CNRS UMR 6625\\
263 avenue du Général Leclerc, 35042 Rennes cedex, France.\\
\texttt{vincent.guirardel@univ-rennes1.fr}

\end{flushleft}
     \end{document}